\documentclass{article}
\usepackage[dvipdfmx, hiresbb]{graphicx}
\usepackage{amsmath,amsfonts,latexsym,graphicx,amssymb,makeidx,amsthm,txfonts}
\usepackage[tips,matrix,arrow,frame]{xy}
\usepackage[dvipddfm]{pict2e}
\newtheorem{Def}{Definition}[section]
\newtheorem{Thm}[Def]{Theorem}
\newtheorem{Lem}[Def]{Lemma}
\newtheorem{Cor}[Def]{Corollary}
\newtheorem{Exm}[Def]{Example}
\newtheorem{Rem}[Def]{Remark}

\usepackage[dvipdfmx]{graphicx,color}
 
\usepackage[dvipdfmx]{graphicx,color}

\begin{document}
\title{
A function from Stern's diatomic sequence, and its properties}
\author{Yasuhisa Yamada}
\maketitle

\begin{abstract}
We define a function by refining Stern's diatomic sequence.
We name it the {\it assembly function}.
It is strictly increasing continuous.
The first and the second main theorems are on an action to the function.
The third theorem is on differentiability of the function at rational points.

\bigskip


\end{abstract}

\tableofcontents

\newpage

\section{Introduction}
The Stern's diatomic sequence 
is a sequence $\{a_m \}_{m=0}^{\infty}$ of non-negative integers
defined by 
$
a_0=0,\ a_1=1,\ a_{2m}=a_m,\ a_{2m+1}=a_m+a_{m+1}.
$
M. A. Stern defined it in \cite{Stern}, 
after that several authors have studied it (e.g. \cite{Calkin, Giuli, Northshield}).
In the present paper, we refine $a_m$ as $[2^n:m] \, (m, n \in \mathbb{Z}_{\ge 0}, 0 \le m \le 2^n)$, 
which is called the {\it Stern's diatomic integer } (SDI, for short) 
with {\it depth} $n$ and {\it order} $m$.
An SDI $[2^n:m]$ is nothing but $a_m$ as an integer.
We arrange SDIs as vertices of a fixed infinite graph.
The resulting one is called the Stern's diatomic table (SDT, for short) 
(cf. Definition \ref{Def_SDA} and Fig. 2-1).
Precisely, $[2^n:m]$ is situated on the $n$-th line (the depth $n$) and 
the order $m$ from the left in the SDT.
As a result, an SDI $[2^n:m]$ is an integer $a_m$ with information of the place in the table.
In Section 2, we give definitions and some properties of SDI and SDT. 

In Section 3, we give a definition and some properties of {\it design}.
For $[2^n:m]$, we define a binary number presentation of it 
as $\{m\}_n=d_1d_2 \cdots d_n$ such that $d_i=0$ or $1\ (1 \le i \le n)$, 
and $\sum_{i=1}^n2^{n-i}d_i=m$, 
and we call $\{m\}_n$ the {\it design} of $[2^n:m]$.
Then we also denote $[2^n:m]$ by $[\{m\}_n]$. 
We often regard $\{m\}_n$ as just a word (not a number).

In Section 4, we give a continued fraction presentation for an SDI
via the design of the SDI by using continuant.
We show that $[2^n:m]/[2^n:2^n-m]$ is equal to the continued fraction 
determined from the design of $[2^n:m]$ (Theorem \ref{Lem_FRC_Product2}).

In Section 5, we define and discuss a Stern's diatomic matrix.
Let $U(2^n:m)$ be a $2 \times 2$ matrix determined from $[2^n:m]$
with the (1, 2)-entry $[2^n:m]$.
We call $U(2^n:m)$ the {\it Stern's diatomic matrix} (SDM, for short) of $[2^n:m]$.
We also denote $U(2^n:m)$ by $U(\{m\}_n)$, and call $\{m\}_n$ the {\it design} of $U(2^n:m)$. 
Then it is a unimodular matrix over $\mathbb{Z}$ with non-negative entries (Theorem \ref{Thm_det=1}).
Conversely, every unimodular matrix over $\mathbb{Z}$ with non-negative entries is 
of the form $U(2^n:m)$ (Theorem \ref{Thm_Modular Matrix_Depth Order}).

In Section 6, we give a definition and some properties of the {\it assembly function}.
Let $S$ be the set of rational numbers of the form $m/2^n$, 
where $m$ and $n$ are two non-negative integers with $0 \le m \le 2^n-1$.
Then we define a map 
$
\mathcal{A} : S \, \rightarrow \, \mathbb{R}_{ \ge 0}
$
by
\[
\mathcal{A}\,\bigg(\frac{m}{2^n}\bigg)=\frac{[2^n:m]}{[2^n:2^n-m]}.
\]
By Theorem \ref{Thm_FRC_Order} (Corollary \ref{Thm_RRC_Continued Fraction}), 
the image of $\mathcal{A}$ is $\mathbb{Q}_{ \ge 0}$.
Since $\mathcal{A}$ is strictly increasing,
and $S$ and $\mathbb{Q}_{\ge0}$ are dense in $[0, 1)$ and $\mathbb{R}_{\ge0}$ respectively,
we can define 
$
\mathcal{A} : [0, 1) \, \rightarrow \, \mathbb{R}_{\ge0}
$
by completing the original one.
Then $\mathcal{A}$ is a strictly increasing continuous function (Theorem \ref{KIJLMMKU}).
We call $\mathcal{A}$ the {\it assembly function}. 
By using the assembly function, 
we prove two main theorems ``Design Composition Theorem I'' 
(Theorem \ref{Thm_The Deviding Formula of Ichio Function})
and ``Design Composition Theorem II'' 
(Theorem \ref{Thm_The Deviding Formula of Matrix}).
Design Composition Theorem I shows an action of SDM (SDI) to the assembly function.
For two designs $\{m\}_n$ and $\{m'\}_{n'}$, 
we define the composition $\{m\}_n \cdot \{m'\}_{n'}$ of them as dyadic words (Definition \ref{SJEDGFSTYT}).
Then the set of designs including the empty design has a free monoid structure whose generators are $0$ and $1$,
and the unit is the empty word.
``Design Composition Theorem II'' gives us a homomorphism of monoids 
from the set of designs to the set of $2 \times 2$ unimodular matrices with non-negative entries, 
which is a monoid representation.
It is not so hard to see that the assembly function is essentially equivalent to 
the inverse function of the Minkowski's question mark function (\cite{Conley}).


In Section 7, we define a {\it periodic design} which is an infinite design.
Then the set of periodic designs is mapped to the set of 
quadratic irrational numbers via the assembly function (Theorem 7.4).
We show an essentially equivalent result to Legendre theorem
on an expression of the square of a rational number via design (Corollary \ref{DDDWHKHCDGH}).

In Section 8, as an application of previous sections, 
we discuss differentiability of the assembly function at rational points.
Lebesgue's theorem says that a monotonely increasing continuous function 
is differentiable at almost every points.
Since the assembly function is strictly increasing continuous, 
it is differentiable at almost every points.
The third main theorem is that at a rational point if the assembly function is differentiable,
then the derivation vanishes, and 
if the assembly function is not differentiable,
then the derivation is $\infty$ (Theorem \ref{main3}).
In our next paper \cite{Yamada2}, we show that at any point,
the derivation is $0$ or $\infty$, and gave a necessary and sufficient condition
for differentiability on a rational point.


In our forthcoming papers, we apply the present results to solve the Markov Conjecture which is one of important Diophantine problems.

\section{Stern's diatomic integer and Stern's diatomic table}

The {\it Stern's diatomic sequence} is defined by
$
a_0=0,\ a_1=1,\ a_{2m}=a_m,\ a_{2m+1}=a_m+a_{m+1}.
$
We refine it by the following definition.

\begin{Def}
\label{Def_SDA}
{\rm For two non-negative integers $m$ and $n$ with $0 \le m \le 2^n$, we define an integer $[2^n:m]$ by the following rules:} 
\begin{enumerate}
{\rm \item $[2^0:0]=0, \,\,\, [2^0:1]=1, $
\item $[2^{n+1}:2m]=[2^n:m]  \quad ( 0 \le m \le 2^n ),$
\item $[2^{n+1}:2m+1]=[2^n:m]+[2^n:m+1] \quad ( 0 \le m \le 2^n-1 ).$} 
\end{enumerate}
{\rm We call $[2^n:m]$ the {\color{black}{\it Stern's diatomic integer}} (SDI, for short) with {\color{black}{\it depth}} $n$ and {\color{black}{\it order}} $m$.
SDIs are expressed in Figure 2-1.We call this table of SDIs the {\color{black}{\it Stern's diatomic table}} (SDT, for short).}
\end{Def}

\begin{figure}[h]
\[
\UseTips
\newdir{ >}{!/-5pt/\dir{>}}
\xymatrix @=0.15pc @*[c]
{\underset{(0)}{[2^0:0]} \ar@{-}[rrrrrrrr] \ar@{->}[ddd] & & & & \ar@{->}[ddd] & & & & \underset{(1)}{[2^0:1]} \ar@{->}[ddd] \\
\vspace{2\baselineskip} \\
& & & & & & & & \\
\underset{(0)}{[2^1:0]} \ar@{-}[rrrr] \ar@{->}[ddd] & & \ar@{->}[ddd] & & \underset{(1)}{[2^1:1]} \ar@{-}[rrrr] \ar@{->}[ddd] & & \ar@{->}[ddd] & & \underset{(1)}{[2^1:2]} \ar@{->}[ddd]\\
\vspace{2\baselineskip} \\
& & & & & & & & \\
\underset{(0)}{[2^2:0]} \ar@{-}[rr] \ar@{->}[ddd] & \ar@{->}[ddd] & \underset{(1)}{[2^2:1]} \ar@{-}[rr] \ar@{->}[ddd] & \ar@{->}[ddd] & \underset{(1)}{[2^2:2]} \ar@{->}[ddd] \ar@{-}[rr] \ar@{->}[ddd] & \ar@{->}[ddd] & \underset{(2)}{[2^2:3]} \ar@{-}[rr] \ar@{->}[ddd] & \ar@{->}[ddd] & \underset{(1)}{[2^2:4]} \ar@{->}[ddd] \\
\vspace{2\baselineskip} \\
& & & & & & & & \\
\underset{(0)}{[2^3:0]} & \underset{(1)}{[2^3:1]} & \underset{(1)}{[2^3:2]} & \underset{(2)}{[2^3:3]} & \underset{(1)}{[2^3:4]} & \underset{(3)}{[2^3:5]} & \underset{(2)}{[2^3:6]} & \underset{(3)}{[2^3:7]} & \underset{(1)}{[2^3:8]} }
\]
\begin{center}
\small{Fig. 2-1}
\end{center}
\end{figure}

For two non-negative integers $m$ and $n$ with $0 \le m \le 2^n$, we have $[2^n:m]=a_m,$
but $[2^n:m]$ is an integer $a_m$ with information of the place in the table SDT.  
From the definition, we can immediately have the following relations.

\begin{enumerate}
\item $[2^{n+1}:2m]=[2^n:m] \quad (0 \le m \le 2^n)$\,,
\item $[2^{n+1}:m]=[2^n:m] \quad (0 \le m \le 2^n)$\,,
\item $[2^n:2m]=[2^n:m] \quad (0 \le m \le 2^{n-1})$\,. 
\end{enumerate}

\begin{Thm}
\label{Thm_det=1}
Let $m$ and $n$ be two non-negative integers with $0 \le m \le 2^n-1$. Then, we have
\[
[2^n:m+1][2^n:2^n-m]-[2^n:m][2^n:2^n-(m+1)]\!=\!1.
\tag{2:1} \label{det=1}
\]
\end{Thm}

\begin{proof}
We prove (\ref{det=1}) by induction on $n$.
If $n=0$,  then the equation clearly holds. 
We suppose that the statement holds for the case $n \ge 0$. 
We show the case $n+1$. 
If $m$ is an even integer such that $m=2l \,\, (l=0, 1, 2, \ldots, 2^n-1)$, by the assumption, we have  
\[
\begin{split}
&[2^{n+1}:m+1][2^{n+1}:2^{n+1}-m]-[2^{n+1}:m][2^{n+1}:2^{n+1}-(m+1)] \\
&=[2^{n+1}:2l+1][2^{n+1}:2^{n+1}-2l]-[2^{n+1}:2l][2^{n+1}:2^{n+1}-(2l+1)] \\
&=\big( \, [2^n:l]+[2^n:l+1] \, \big)[2^n:2^n-l]-[2^n:l]\big( \, [2^n:2^n-l]+[2^n:2^n-(l+1)] \, \big) \qquad \qquad \qquad \qquad \qquad \qquad \qquad \qquad \qquad \\
&=[2^n:l+1][2^n:2^n-l]-[2^n:l][2^n:2^n-(l+1)]=1. 
\end{split}
\]
If $m$ is an odd integer such that $m=2l+1 \,\, (l=0, 1, 2, \ldots, 2^n-1)$, by the assumption, we have 
\[
\begin{split}
&[2^{n+1}:m+1][2^{n+1}:2^{n+1}-m]-[2^{n+1}:m][2^{n+1}:2^{n+1}-(m+1)] \\
&=[2^{n+1}:2l+2][2^{n+1}:2^{n+1}-(2l+1)]-[2^{n+1}:2l+1][2^{n+1}:2^{n+1}-(2l+2)]  \qquad \qquad \qquad \qquad \qquad \qquad \\ 
&=[2^n\!:l+1]\big(\,[2^n\!:\!2^n-l]+[2^n\!:\!2^n-(l+1)]\,\big)-\big(\,[2^n\!:l]+[2^n\!:l+1]\,\big)[2^n\!:\!2^n-(l+1)] \\
&=[2^n:l+1][2^n:2^n-l]-[2^n:l][2^n:2^n-(l+1)]=1. 
\end{split}
\]
Therefore we have the result.
\end{proof} 

From this theorem, we immediately have the following result.

\begin{Cor}
\label{CorCoprime}
Let $m$ and $n$ be two non-negative integers with $0 \le m \le 2^n-1$. Then, we have \\
{\rm (1)} \, $[2^n:m],\, [2^n:m+1]$ are coprime. \\[3pt]
{\rm (2)} \, $[2^n:m],\, [2^n:2^n-m]$ are coprime.
\end{Cor}

It is essential to find the values of $[2^n:m]$ with odd order since $[2^n:2m]=[2^n:m]$.

\begin{Lem}
\label{Lem_FRC-formula1}
For any non-negative integer $a$ and two consecutive SDIs $[2^n\!:\!m], \,[2^n\!:\!m+1] \,(0 \le m \le 2^n-1)$, \\
{\rm (1)} $a[2^n :m+1]+[2^n :m]=[2^{n+a}:2^am+2^a-1]$. In particular, $[2^a:2^a-1]=a \,\, (m=n=0)$. \\[3pt]
{\rm (2)} $[2^n :m+1]+a[2^n :m]=[2^{n+a}:2^am+1]$.
\end{Lem}

\begin{proof}
We prove (1) by induction on $a$.
It obviously holds for $a=0$. 
Suppose (1) holds for some non-negative integer $a$.
Then we observe that
\[
\begin{split}
&(a+1)[2^n :m+1]+[2^n :m]=a[2^n :m+1]+[2^n :m+1]+[2^n :m] \\
&=[2^{n+a}:2^am+2^a-1]+[2^{n+a}:2^am+2^a]=[2^{n+(a+1)}:2^{a+1}m+2^{a+1}-1], \\
\end{split}
\]
which completes the proof. (2) can be proved in the same way as (1).
\end{proof}

\begin{Def}
\label{Def_Partial Quotients}
{\rm For two positive coprime integers $a$ and $b$, we apply the Euclidean algorithm as 
\[
c_{i-2}=c_{i-1}r_i+c_i \,\,\,\,\, (i=0,1, \ldots, t-1),
\tag{2:2} \label{Euclid's algorithm}
\]
where $t \ge 1, c_{-2}=a, c_{-1}=b$, and $c_i$ and $r_i \,\,\, (i=0, 1, \ldots, t-1)$ are integers such that 
$r_0 \ge 0, r_i \ge 1 \,\,\,(i=1, 2, \ldots, t-1), 0 \le c_i < c_{i-1} \,\,\, (i=0, 1, \ldots, t-1), c_{t-2}=1$ and $c_{t-1}=0$.
Then we call the sequence $r_0, r_1, r_2, \ldots, r_{t-1}$ the {\it partial quotients} or simply the {\it quotients} of $a$ generated by $b$,
and $t$ the {\it length} of the partial quotients}.
\end{Def}

\begin{Lem}
\label{KRPFHWK}
Under the situation in Definition \ref{Def_Partial Quotients}, we have the following: \\
{\rm (1)} \, $r_0=0$ if and only if $a<b$. Equivalently, $r_0 \ge 1$ if and only if $a \ge b$. \\ 
{\rm (2)} \, $t=1$ if and only if $b=1$ $($then $r_0=a, c_0=0$$)$. \\
{\rm (3)} \, $r_{t-1}=1$ if and only if $a=b=1$ $($then $t=1$$)$. Equivalently, $r_{t-1} \ge 2$ if and only if $(a, b) \neq (1, 1)$. \\
{\rm (4)} \, For a sequence of integers $r_0, r_1, \ldots, r_{t-1}$ with $r_0 \ge 0, r_i \ge 1 \, (i=1, 2, \ldots, t-2)$ and $r_{t-1} \ge 2$ $($including the cases $t=1$ and $2$$)$,
there exists a unique pair of two positive coprime integers $a$ and $b$ such that $r_0, r_1, \cdots, r_{t-1}$ are the partial quotients of $a$ generated by $b$ 
$($then $(a, b) \neq (1, 1)$$)$. \\
{\rm (5)} \, $c_{i-2}$ and $c_{i-1} (i=0, 1, \ldots, t)$ are coprime. \\
{\rm (6)} \, A sequence $r_i, r_{i+1}, \cdots, r_{t-1} \, (i=0, 1, \ldots, t-1)$ is the partial quotients of $c_{i-2}$ generated by $c_{i-1}$ with length $t-i$.
\end{Lem}

\begin{proof}
Since (1), (4), (5) and (6) are not hard to see, we do not give the proofs of them. We only prove (2) and (3). \\
(2) Suppose $t=1$. Then the partial quotients of $a$ generated by $b$ is $r_0$. Since $a$ and $b$ are coprime, we have $b=1$. 
Conversely, suppose $b=1$. Then the partial quotients of $a$ generated by $b$ is $r_0$, and $t=1$. \\
(3) Suppose $r_{t-1}=1$. Then we have $c_{t-3}=c_{t-2}=1$ and $c_{t-1}=0$. Since $c_{t-2}<c_{t-3}$ if $t \ge 2$, we have $t=1$, and $a=b=1$. 
Conversely, suppose $a=b=1$. Then the partial quotients of $a$ generated by $b$ is $r_0=1 \, (t=1)$.
\end{proof}

\begin{Def}
\label{FLAQHETDH}
{\rm
Under the situation in Lemma \ref{KRPFHWK} (4), for the sequence $r_0, r_1, \ldots, r_{t-1}$, 
a pair of integers $(a, b)$ is called the {\color{black}{\it realizing pair}} of the sequence. 
For the sequence $r_0=1$, the realizing pair of the sequence is $(1, 1)$.
}
\end{Def}

We generalize Lemma \ref{Lem_FRC-formula1}.

\begin{Lem}
\label{Lem_FRC_formula 2}
Let $a$ and $b$ be positive coprime integers, 
and $ r_0, r_1, \ldots, r_{t-1} $ the partial quotients of $a$ generated by $b$.
For any consecutive SDIs $[2^n:m]$ and $[2^n:m+1]$, \\
{\rm (1)} \, $a[2^n :m+1]+b[2^n :m]=[2^{N_1}:M_1]$, \,{\it where} 
\[
N_1=n+\sum_{i=0}^{t-1}r_i \,\,\,\,\, {\it and} \,\,\,\,\, M_1=2^{ \sum_{i=0}^{t-1}r_i}m+\sum_{j=0}^{t-1}(-1)^j2^{ \sum_{i=j}^{t-1}r_i}+(-1)^t.
\]
{\rm (2)} \, $b[2^n :m+1]+a[2^n :m]=[2^{N_2}:M_2]$, \,{\it where} 
\[
N_2=N_1=n+\sum_{i=0}^{t-1}r_i \,\,\,\,\, {\it and} \,\,\,\,\, 
M_2=
\left\{
\begin{split}
&  2^{ \sum_{i=0}^{t-1}r_i}m+\sum_{j=1}^{t-1}(-1)^{j-1}2^{ \sum_{i=j}^{t-1}r_i}+(-1)^{t-1} \,\,\,\,\, (t \ge 2), \\
&  2^{r_0}m+1 \,\,\,\,\, (t=1). 
\end{split}
\right.
\]
\end{Lem}

\begin{proof}
We prove by induction on $t$. 
If $t=1$, then $b=1$. By Lemma \ref{Lem_FRC-formula1}, we have the result. 
Suppose that $t \ge 2$, and that the statement holds for any two positive coprime integers $a'$ and $b'$ such that 
the length of the partial quotients of $a'$ generated by $b'$ is less than or equal to $t-1$.
If we particularly take $a'=c_{-1}(=b)$ and $b'=c_0$, 
then the partial quotients of $a'$ generated by $b'$ is $r_1, r_2, \ldots, r_{t-1}$ with length $t-1$,
and we have, by the assumption,    
\begin{enumerate}
{\rm 
\item $b[2^{n'} :m'+1]+c_0[2^{n'} :m']=[2^{N_1'}:{M_1}\!']$, \, where 
\[
N_1'=n'+\sum_{i=1}^{t-1}r_i \,\,\,\,\, {\rm and} \,\,\,\,\, {M_1}\!'=2^{ \sum_{i=1}^{t-1}r_i}m'+\sum_{j=1}^{t-1}(-1)^{j-1}2^{ \sum_{i=j}^{t-1}r_i}+(-1)^{t-1}.
\]
\item $c_0[2^{n'} :m'+1]+b[2^{n'} :m']=[2^{N_2'}:{M_2}\!']$, \, where 
\[
N_2'=N_1'=n'+\sum_{i=1}^{t-1}r_i \,\,\,\,\, {\rm and} \,\,\,\,\, 
{M_2}\!'=
\left\{
\begin{split}
&   2^{ \sum_{i=1}^{t-1}r_i}m'+\sum_{j=2}^{t-1}(-1)^j2^{ \sum_{i=j}^{t-1}r_i}+(-1)^t \,\,\,\,\, (t \ge 3), \\
&  2^{r_1}m'+1 \,\,\,\,\, (t=2). 
\end{split}
\right.
\]
}
\end{enumerate}
Suppose $a$ and $b$ are positive coprime integers, and that the length of the partial quotients of $a$ generated by $b$ is equal to $t$.   
Note that $a=b r_0+c_0$.
\[
\begin{split}
(1) \,\,\, a[2^n :m+1]+b[2^n :m]=&\big (b r_0+c_0 \big )[2^n :m+1]+b[2^n :m] \qquad \qquad \qquad \quad \\[3pt]
=&c_0[2^n :m+1]+b \big (\, r_0[2^n :m+1]+[2^n:m] \, \big ) \\[3pt]
=&c_0[2^{n+r_0} :2^{r_0}m+2^{r_0}]+b[2^{n+r_0} :2^{r_0}m+2^{r_0}-1].
\end{split}
\]
We set $n'=n+r_0$ and $m'=2^{r_0}m+2^{r_0}-1$. Then,  
$
a[2^n :m+1]+b[2^n :m]\!=\![2^{N_2'}:{M_2}\!'],\\
\text{where}
$
\[
N_2'=n+\sum_{i=0}^{t-1}r_i=N_1 \,\,\,\,\, {\rm and} \,\,\,\,\, {M_2}\!'=2^{ \sum_{i=0}^{t-1}r_i}m+\sum_{j=0}^{t-1}(-1)^j2^{ \sum_{i=j}^{t-1}r_i}+(-1)^t=M_1.
\]
\[
\begin{split}
(2) \,\,\, b[2^n :m+1]+a[2^n :m]= &b[2^n :m+1]+\big (b r_0+c_0 \big )[2^n :m] \qquad \qquad \qquad \quad \\[3pt]
=&b \big ( \, [2^n :m+1]+r_0[2^n:m] \, \big )+c_0[2^n :m] \\[3pt]
=&b[2^{n+r_0} :2^{r_0}m+1]+c_0[2^{n+r_0} :2^{r_0}m].
\end{split}
\]
We set $n'=n+r_0$ and $m'=2^{r_0}m$. Then we have $b[2^n :m+1]+a[2^n :m]=[2^{N_1'}:{M_1}\!']$, where
\[
N_1'=n+\sum_{i=0}^{t-1}r_i=N_2 \,\,\,\,\, {\rm and} \,\,\,\,\, {M_1}\!'=2^{ \sum_{i=0}^{t-1}r_i}m+\sum_{j=1}^{t-1}(-1)^{j-1}2^{ \sum_{i=j}^{t-1}r_i}+(-1)^{t-1}=M_2.
\]
Therefore we have the result.
\end{proof}



\section{Design of Stern's diatomic integer}

In this section, we define a notion of {\it design}, and
we correspond it to an SDI.

\begin{Def} 
{\rm
For two non-negative integers $m$ and $n$ such that $0 \le m \le 2^n$, 
we introduce a symbol $\{m\}_n$,
and call it the {\it design} of $m$ with length $n$,
and $m$ the {\it design number} of $\{m\}_n$. 
If $0 \le m \le 2^n-1$, we call $\{m\}_n$ a {\it finite design} as a general term.
If $m=2^n$, we call $\{m\}_n$ the {\it terminal design} with length $n$. \\ 
(1) If $\{m\}_n$ is a finite design with $n \ge 1$, 
by using a finite dyadic word $d_1d_2 \cdots d_n$, where $d_i=0$ or $1 (i=1, 2, \ldots, n)$ 
and $m=\sum_{i=1}^{n}2^{n-i}d_i$ (i.e. regarding the finite dyadic word as a binary number), 
we denote $\{m\}_n=d_1d_2 \cdots d_n$.
If $\{m\}_n$ is a terminal design with $n \ge 1$,
we denote $\{ 2^n \}_{\tiny n}=1{\underbrace{0 \cdots 00}_{\tiny n}}{}_{\tau}$, 
where $\underbrace{ \cdots }_{\tiny k}$ implies that the number of words is $k$. \\
(2) In the case that $n=0$,
we define $\{0\}_{\tiny 0}=\varepsilon$ and $\{ 2^0 \}_{\tiny 0}=1_{\tau}$, 
where $\varepsilon$ implies the empty word,
and specially call $\{0\}_{\tiny 0}$ the {\it empty design} with length 0.
$\{ 2^0 \}_{\tiny 0}$ is the {\it terminal design} with length $0$. 

We can denote finite designs by
\[
\{m\}_n=
          \underbrace{1 \cdots 11}_{\tiny k_0}
          \underbrace{0 \cdots 00}_{\tiny k_1}
          \cdots 
          \underbrace{0 \cdots 00}_{\tiny k_{l-2}}
          \underbrace{1 \cdots 11}_{\tiny k_{l-1}},
\]
where $l$ is odd, $k_0 \ge 0, k_i \ge 1 \, (i=1, 2, \ldots, l-2), \, k_{l-1} \ge 0$ and $n=\sum_{i=0}^{l-1}k_i$.
Then we have $m=\sum_{j=0}^{l-1}(-1)^j2^{ \sum_{i=j}^{l-1}k_i}-1$. 
We note that if $l=1$ and $k_0=0$, the right-hand side of the equation denotes $\varepsilon \, (=\{0\}_0)$. 

We regard $\{m\}_n$ as the place in the SDT. That is, we denote
$
[\{m\}_n]=[2^n:m],
$ 
and call $\{m\}_n$ the {\it design} of $[2^n:m]$. 
Then Fig.2-1 is represented as Fig.3-1. 
We emphasize that $\{m\}_n$ has information of the dyadic word presentation of $m$ with length $n$
and the place in the SDT.
When the length $n$ is clear by context, we sometimes denote $\{m\}_n$ by $\{m\}$ or $m$ simply.
We define an {\it infinite design} in Section 6. \\[-30pt]
}
\end{Def}

\begin{figure}[!h]
\begin{center}
\[
\UseTips
\newdir{ >}{!/-5pt/\dir{>}}
\xymatrix @=0.15pc @*[c]
{\underset{(0)}{[\,\varepsilon \,]} \ar@{-}[rrrrrrrr] \ar@{->}[ddd] & & & & \ar@{->}[ddd] & & & & \underset{(1)}{[1_{\tau}]} \ar@{->}[ddd] \\
\vspace{2\baselineskip} \\
& & & & & & & & \\
\underset{(0)}{[0]} \ar@{-}[rrrr] \ar@{->}[ddd] & & \ar@{->}[ddd] & & \underset{(1)}{[1]} \ar@{-}[rrrr] \ar@{->}[ddd] & & \ar@{->}[ddd] & & \underset{(1)}{[10_{\tau}]} \ar@{->}[ddd]\\
\vspace{2\baselineskip} \\
& & & & & & & & \\
\underset{(0)}{[00]} \ar@{-}[rr] \ar@{->}[ddd] & \ar@{->}[ddd] & \underset{(1)}{[01]} \ar@{-}[rr] \ar@{->}[ddd] & \ar@{->}[ddd] & \underset{(1)}{[10]} \ar@{->}[ddd] \ar@{-}[rr] \ar@{->}[ddd] & \ar@{->}[ddd] & \underset{(2)}{[11]} \ar@{-}[rr] \ar@{->}[ddd] & \ar@{->}[ddd] & \underset{(1)}{[100_{\tau}]} \ar@{->}[ddd] \\
\vspace{2\baselineskip} \\
& & & & & & & & \\
\underset{(0)}{[000]} & \underset{(1)}{[001]} & \underset{(1)}{[010]} & \underset{(2)}{[011]} & \underset{(1)}{[100]} & \underset{(3)}{[101]} & \underset{(2)}{[110]} & \underset{(3)}{[111]} & \underset{(1)}{[1000_{\tau}]} }
\]
\small{Fig. 3-1} 
\end{center}
\end{figure}

\medskip

\begin{Exm} \label{JDTAB} 
\[
\begin{array}{l l l l l}
\{0\}_{\tiny 0}=\varepsilon,  &\{1\}_{\tiny 0}=1_{\tau}, & \{0 \}_{\tiny 1}=0, & \{1 \}_{\tiny 1}=1, & \{2 \}_{\tiny 1}=10_{\tau}, \\
\{0 \}_{\tiny 2}=00, &\{1 \}_{\tiny 2}=01,& \{2 \}_{\tiny 2}=10, & \{3 \}_{\tiny 2}=11, & \{4 \}_{\tiny 2}=100_{\tau}.
\end{array}
\]
\end{Exm}

We need the following lemma. However, we do not give the proofs of them because they are easy to see.

\begin{Lem} \label{ZZWHOPN}
For a finite design 
$
\{m\}_n=
          \underbrace{1 \cdots 11}_{\tiny k_0}
          \underbrace{0 \cdots 00}_{\tiny k_1}
          \cdots 
          \underbrace{0 \cdots 00}_{\tiny k_{l-2}}
          \underbrace{1 \cdots 11}_{\tiny k_{l-1}},
$ \\
{\rm (1)} \, $m$ is odd if and only if $k_{l-1} \ge 1$. \\
{\rm (2)} \, $2^{n-1} \le m \le 2^n-1$ if and only if $k_0 \ge 1$. \\
{\rm (3)} \, $m \equiv 1 \, (mod 4)$ if and only if $k_{l-1}=1$, and $ m \equiv 3 \, (mod 4)$ if and only if $k_{l-1} \ge 2$.
\end{Lem}

\begin{Def} \label{HIJJYQWJS}
{\rm
For a finite design 
$
\{m\}_n=
          \underbrace{1 \cdots 11}_{\tiny k_0}
          \underbrace{0 \cdots 00}_{\tiny k_1}
          \cdots 
          \underbrace{0 \cdots 00}_{\tiny k_{l-2}}
          \underbrace{1 \cdots 11}_{\tiny k_{l-1}},
$
we call $\{m\}_n$ a {\color{black}{\it reduced design}} if it satisfies (1) in Lemma \ref{ZZWHOPN} (i.e. $m$ is odd), 
and $\{m\}_n$ a {\color{black}{\it primitive design}} if it satisfies (1) and (2) in Lemma \ref{ZZWHOPN}. 
Additionally, we consider that $\{0\}_0 \, (=\varepsilon)$ and $\{1\}_0 \, (=1_{\tau})$ are reduced designs. 
}
\end{Def}

\begin{Def} \label{SJEFSJQR}
{\rm 
Let $\{m\}_n$ be a design with length $n$.
Then, by using $\{m\}_n$, we denote other designs with length $n$ by the following rules: \\
(1) \, For an integer $m'$ with $0 \le m' \le 2^n-m$, 
\[
\{m\}_n+m'=\{m\}_n+\{m'\}_n=\{m+m'\}_n. 
\]
(2) \, For an integer $m'$ with $0 \le m' \le m$, 
\[
\{m\}_n-m'=\{m\}_n-\{m'\}_n=\{m-m'\}_n. 
\]
(3) \, $2^n-\{m\}_n=\{2^n\}_n-\{m\}_n=\{2^n-m\}_n$, and call it the {\it conjugate design} of $\{m\}_n$. 

For example, if 
$
\{m\}_n=
          \underbrace{1 \cdots 11}_{\tiny k_0}
          \underbrace{0 \cdots 00}_{\tiny k_1}
          \cdots 
          \underbrace{0 \cdots 00}_{\tiny k_{l-2}}
          \underbrace{1 \cdots 11}_{\tiny k_{l-1}}
$
is a finite reduced design with length $n \ge 1$, then we have
\[
2^n-\{m\}_n=
          \underbrace{0 \cdots 00}_{\tiny k_0}
          \underbrace{1 \cdots 11}_{\tiny k_1}
          \cdots 
          \underbrace{1 \cdots 11}_{\tiny k_{l-2}}
          \underbrace{0 \cdots 01}_{\tiny k_{l-1}}.
\]
If $\{2^n\}_n$ is a terminal design with length $n \ge 1$, then we have 
$
2^n-\{2^n\}_n=\{0\}_n= \underbrace{0 \cdots 00}_{\tiny n}.
$
In particular, the conjugate design of $\{1\}_0$ is the empty design, 
and conversely the conjugate design of the empty design is $2^0-\{0\}_0=1_{\tau}.$
}
\end{Def}

\begin{Def} \label{GSIEVGTHPA}
{\rm
Let $a$ and $b$ be two positive coprime integers,
and $ r_0, r_1, r_2, \ldots, r_{t-1} $ the partial quotients of $a$ generated by $b$.
Then we define the following reduced design $D$:

\[
D=
\left\{
\begin{split}
&  \qquad \qquad \qquad \,\,\,\, 1 \qquad \qquad \qquad \,\,\,\,\,\,\,\, (\,(a, b)=(1, 1)\,), \\
&  \underbrace{1 \cdots 11}_{\tiny r_0}
    \underbrace{0 \cdots 00}_{\tiny r_1}
    \cdots
    \underbrace{0 \cdots 00}_{\tiny r_{t-2}}
    \underbrace{1 \cdots 11}_{\tiny r_{t-1}} \,\,\,\,\, (\,(a, b)\neq(1, 1), \,\, t\!:\!odd \,), \\
&  \underbrace{1 \cdots 11}_{\tiny r_0}
    \underbrace{0 \cdots 00}_{\tiny r_1}
    \cdots
    \underbrace{1 \cdots 11}_{\tiny r_{t-2}}
    \underbrace{0 \cdots 01}_{\tiny r_{t-1}} \,\,\,\,\, (\,(a, b)\neq(1, 1), \,\, t\!:\!even \,). 
\end{split}
\right.
\]
We call $D$ the {\it Euclidean design} of $a$ generated by $b$.
If $D=\{m\}_n$, then we call $m$ the {\it design number} of $a$ generated by $b$, where $n=\sum_{i=0}^{t-1}r_i$.
In particular, if $a=b=1$, then $D=\{1\}_1$.
We note that
by Lemma \ref{KRPFHWK} (3), if $(a, b) \neq (1, 1)$, then $r_{t-1} \ge 2$ and  $m$ is an odd integer with $1 \le m \le 2^n-1$.
By Lemma \ref{ZZWHOPN} (3), $m \equiv 3 \, (mod 4)$ if $t$ is odd, and $m \equiv 1 \, (mod 4)$ if $t$ is even. 
}
\end{Def}

\begin{Lem} \label{DEWHTQXXJ}
Under the situation in Definition \ref{GSIEVGTHPA}, we have
\[
m=\sum_{j=0}^{t-1}(-1)^j2^{\sum_{i=j}^{t-1}r_i}+(-1)^t \,\,\,\,\,\, \text{and} \,\,\,\,\,\, 2^n-m=\sum_{j=1}^{t-1}(-1)^{j-1}2^{\sum_{i=j}^{t-1}r_i}+(-1)^{t-1}.
\]
\end{Lem}

The following is the first main theorem of this paper.
Theorem says that the set of pairs of two positive coprime integers has one to one correspondence 
with the set of two positive integers $(n, m)$ such that $m$ is odd and $1 \le m \le 2^n-1$ via design representations.
We call the theorem the Design Representation Theorem.

\begin{Thm}{\rm (Design Representation Theorem)}
\label{Thm_FRC_Order} \\
{\rm (1)} \, Let $a$ and $b$ be positive coprime integers, and $\{m\}_n$ the Euclidean design of $a$ generated by $b$.
Then we have $a=[\{m\}_n]=[2^n:m]$ and $b=[2^n-\{m\}_n]=[2^n:2^n-m]$.
In particular, if $(a, b)=(1, 1)$, then $(m, n)=(1, 1)$. \\
{\rm (2)} \, Let $m$ and $n$ be two positive integers such that $m$ is an odd integer with $1 \le m \le 2^n-1$ and $(m, n) \neq (1, 1)$.
We suppose
$
\{m\}_n=  
    \underbrace{1 \cdots 11}_{\tiny k_0}
    \underbrace{0 \cdots 00}_{\tiny k_1}
    \cdots
    \underbrace{0 \cdots 00}_{\tiny k_{l-2}}
    \underbrace{1 \cdots 11}_{\tiny k_{l-1}},  
$
where $l$ is odd, $k_0 \ge 0$ and $k_i \ge 1 \, (i=1, 2, \ldots, l-1)$. \\
{\rm (i)} \, If $m \equiv 3 \, (mod 4) \, (i.e. \, k_{l-1} \ge 2)$, then by taking $t=l$ and $r_i=k_i \, (i=0, 1, \ldots, t-1)$,
we have the realizing pair $(a, b)$ of $r_0, r_1, \ldots, r_{t-1}$, and $\{m\}_n$ is the Euclidean design of $a$ generated by $b$. \\
{\rm (ii)} \, If $m \equiv 1 \, (mod 4) \, (i.e. \, k_{l-1}=1)$, then by taking $t=l-1$, $r_i=k_i \, (i=0, 1, \ldots, t-2=l-3)$ and $r_{t-1}=k_{l-2}+1$,
we have the realizing pair $(a, b)$ of $r_0, r_1, \ldots, r_{t-1}$, and $\{m\}_n$ is the Euclidean design of $a$ generated by $b$.
\end{Thm}

\begin{proof}
(1) It is easy to see that if $(a, b)=(1, 1)$, then $(m, n)=(1, 1)$. From now on, we suppose $(a, b) \neq (1, 1)$.
Since $[2^0:0]=0$ and $[2^0:1]=1$,
\[
a=a[2^0:1]+b[2^0:0]=[\{m\}_n]=[2^n : m] \\
\]
and
\[
b=b[2^0:1]+a[2^0:0]=[2^n-\{m\}_n]=[2^n : 2^n-m]
\]
by Lemma \ref{Lem_FRC_formula 2} and Lemma \ref{DEWHTQXXJ}. \\
(2) Let $(a, b)$ be the realizing pair of $r_0, r_1, \ldots, r_{t-1}$.
Since $(m, n) \neq (1, 1)$, we have $(a, b) \neq (1, 1)$. Then we have 
\[
m=\sum_{j=0}^{t-1}(-1)^j2^{\sum_{i=j}^{t-1}r_i}+(-1)^t,
\]
and $\{m\}_n$ is the Euclidean design of $a$ generated by $b$ by Lemma \ref{DEWHTQXXJ}.
\end{proof}

Next, we characterize $(a, b)$ corresponding to a primitive design.

\begin{Thm} \label{HQPJJTSKB}
Let $a$ and $b$ be positive coprime integers with $a \ge b$, and $b'=r a+b$ for a positive integer $r$.
Let $D$ and $D'$ be the Euclidean designs of $a$ generated by $b$ and $b'$, respectively.
Then $D$ is a primitive design, $D'= \underbrace{0 \cdots 00}_{\tiny r}D$, and the design numbers of $D$ and $D'$ are the same.
\end{Thm}

\begin{proof}
We set $D=d_1d_2 \cdots d_n$, where $d_i=0$ or $1 \, (i=0, 1, \ldots, n)$. By Lemma \ref{KRPFHWK} (1) and Lemma \ref{ZZWHOPN} (1), $D$ is primitive.
Let $r_0, r_1, \ldots, r_{t-1}$ be the partial quotients of $a$ generated by $b$.
Then the partial quotients of $a$ generated by $b'$ are $0, r, r_0, r_1, \ldots, r_{t-1}$.
Therefore $D'= \underbrace{0 \cdots 00}_{\tiny r}D$, and the design numbers of $D$ and $D'$ are the same.  
\end{proof}

\begin{Cor}
\label{Cor_The Number of Odd Design}
A positive integer $a$ can be represented $\varphi(a)$ distinct ways exactly by primitive designs, where $\varphi(a)$ is the Euler's function.
\end{Cor}

\begin{proof}
By Theorem \ref{HQPJJTSKB}, a primitive design representation is characterized by a positive integer $b$ coprime to $a$ with $a \ge b$. 
\end{proof}

\begin{Exm} \label{TDKDCW}
{\rm We determine the primitive designs such that $12=[2^n : m]=[\{m\}_n]$.
12 has four natural numbers which are prime to 12 and not greater than 12, namely $1, 5, 7$ and $11$. 
$111111111111=\{4095\}_{12}$, $110011=\{51\}_6$, $101101=\{45\}_6$ and $100000000001=\{2049\}_{12}$
are the Euclidean designs of 12 generated by 1, 5, 7 and 11 respectively. 
Hence, we have $[2^{12}:4095]=12,$ $[2^6:51]=12,\, [2^6:45]=12$ and $[2^{12}:2049]=12$.} 
\end{Exm}

\section{Continued fraction presentation of SDI}

In this section, we give a continued fraction presentation for an SDI
via the design presentation of the SDI (Theorem \ref{Lem_FRC_Product2}).
For continued fractions, see \cite{Hardy}.

\begin{Def}
\label{Def_Special Matrix}
{\rm 
The {\it continuant} $[x_0, x_1, \ldots, x_{l-1}]$ with length $l \, (\ge 1)$ \\ is a polynomial with $l$-variables $x_0, x_1, \ldots, x_{l-1}$
defined by the following recursive relations:
\begin{align}
&[x_0]=x_0, \,\, [x_0, x_1]=x_0x_1+1, \notag \\
&[x_0, x_1, \ldots, x_{l-3}, x_{l-2}, x_{l-1}]=[x_0, x_1, \ldots, x_{l-3}, x_{l-2}]\, x_{l-1}+[x_0, x_1, \ldots, x_{l-3}]. \tag{4:1} \label{SXSAWJTDDJ} 
\end{align}
As a special case, for the length $0$, we define $[\varepsilon]=1$.
Then the recursive relation (\ref{SXSAWJTDDJ}) also holds for the case $l=2$.
}
\end{Def}


\begin{Lem}
\label{Lem_Special Matrix}
For the continuant, we have the following formulas.
Basically we assume $l \ge 1$.
\[
\begin{split}
&(1) \,\,\, [x_0, x_1,x_2,\ldots,x_{l-1}]
=\left|
\begin{matrix}
              x_0 \!\!\!\!& 1     \!\!\!\!&        \!\!\!\!&             \!\!\!\!& {\Large \text{$0$}} \\
              -1  \!\!\!\!& x_1  \!\!\!\!& \! 1      \!\!\!\!&             \!\!\!\!&   \\
                  \!\!\!\!& -1    \!\!\!\!& \! x_2    \!\!\!\!& \ddots            &   \\
                  \!\!\!\!&        \!\!\!\!&  \ddots        & \ddots      & 1 \\
{\Large \text{$0$}} \!\!\!\!&      \!\!\!\!&        \!\!\!\!& -1          \!\!\!\!&\! x_{l-1} 
\end{matrix}
\right|. \qquad \qquad \qquad \qquad \qquad \qquad \qquad \qquad 
\end{split}
\]
\[
\begin{split}
&(2) \,\,\, [x_0, x_1, x_2, \ldots, x_{l-1}]=[x_{l-1}, x_{l-2}, x_{l-3}, \ldots, x_0]. \\
&(3) \,\,\, [x_0, x_1, x_2, \ldots, x_{l-1}]=x_0[x_1, x_2, \ldots, x_{l-1}]+[x_2, \ldots, x_{l-1}] \,\,\,\, (l \ge 2).  \qquad \qquad \qquad \qquad \qquad \qquad \\
&(4) \,\,\, [x_0, x_1, x_2, \ldots, x_{l-1}, 1]=[x_0, x_1, x_2, \ldots, x_{l-1}+1]. \\
&(5) \,\,\, [1, x_0, x_1, x_2, \ldots, x_{l-1}]=[1+x_0, x_1, x_2, \ldots, x_{l-1}]. \\
&(6) \,\,\, [x_0, x_1, x_2, \ldots, x_{l-1}, 0]=[x_0, x_1, x_2, \ldots, x_{l-2}]. \\
&(7) \,\,\, [0, x_0, x_1, x_2, \ldots, x_{l-1}]=[x_1, x_2, \ldots, x_{l-1}]. \\
&(8) \,\,\, [x_0, x_1, \ldots, x_{l-2}, x_{l-1}+x]=[x_0, x_1, \ldots, x_{l-2}]x+[x_0, x_1, \ldots, x_{l-2}, x_{l-1}]. 
\end{split}
\]
\end{Lem}

\begin{proof}
(1) We prove by induction on $l$. We set the right-hand side as $\Delta_l$.
It is easy to see $\Delta_1=x_0=[x_0]$ and $\Delta_2=x_0x_1+1=[x_0, x_1]$. 
Suppose that $l \ge 2$, 
and that the statement holds for the cases $l-1$ and $l$.
We show the case $l+1$. 
By expanding on the $(l+1)$-th row of $\Delta_{l+1}$, we have $\Delta_{l+1}=\Delta_lx_l+\Delta_{l-1}$. 
By the recursive relation of the continuant and the assumption, we have $\Delta_{l+1}=[x_0, x_1, x_2, \cdots, x_l]$. \\
(2) By (1), we have the result. \\
(3)-(8) By (1) and (2), we have the results.
\end{proof}

We represent SDIs by using the continuant.
In a continued fraction, we allow computations
$1/\infty=0$ and $1/0=\infty$.

\begin{Thm}
\label{Thm_FRC_Matrix}
For any SDI $[2^n : m]$ with $0 \le m \le 2^n$, 
we denote the design of $[2^n : m]$ by
\[
\{m\}_n=
\underbrace{1 \cdots 11}_{\tiny k_0}
    \underbrace{0 \cdots 00}_{\tiny k_1}
    \cdots 
    \underbrace{0 \cdots 00}_{\tiny k_{l-2}}
    \underbrace{1 \cdots 11}_{\tiny k_{l-1}},
\]
where $l$ is odd, $k_0 \ge 0, k_i \ge 1 \,\,\,(i=1, 2, \ldots, l-2)$ and $k_{l-1} \ge 0$.
Then we have
\[
[2^n : m]=[k_0, k_1, k_2, \ldots, k_{l-1}]. 
\]
In particular, for the case $k_{l-1}=0$ (i.e. $m$ is even), if $l \ge 3$, we have 
\[
[2^n : m]=[k_0, k_1, k_2, \ldots, k_{l-3}].
\]
\end{Thm}

\begin{proof} 
Firstly, we show $[2^n : 2^n] \, (n=0, 1, 2, \ldots)$. 
Then $[2^n : 2^n]$ has a terminal design $\{2^n\}_n=1\underbrace{0 \cdots 00}_{\tiny n}{}_{\tau}$ \, (i.e. $k_0=1, k_1=n, k_2=0$).
Hence we have 
\[
[2^n : 2^n]=1=[1, n, 0]=[k_0, k_1, k_2].
\]

Secondly, we show $[2^n : m] \, (0 \le m \le 2^n-1)$. 
Suppose $l=1$. Then we have $k_0=n, m=2^n-1$ and $[2^n : m]=[2^n : 2^n-1]=n=[k_0]$ by Lemma \ref{Lem_FRC-formula1} (1). 
Therefore the statement holds. If $(m, n)=(1, 1)$, then we have $\{m\}_n=\{1\}_1=1$. Hence the case is included by the case $l=1$.
From now on, we assume $(m, n) \neq (1, 1)$. \\
(1) \, The case $k_{l-1}=0$. 
We set $l=2s-1 \, (s=1, 2, \ldots)$. Only this case, we show by induction on $s$.
The case $l=1$ has already shown.
Next, we assume that the statement holds for $l \le 2s-1$, and we set $l=2s+1 \ge 3$.
Then by the assumption, we have $[2^n : m]=[2^{n-k_{l-2}} : 2^{-k_{l-2}}m]=[k_0, k_1, \cdots, k_{l-3}]$.
By Lemma \ref{Lem_Special Matrix} (6), we have $[2^n : m]=[k_0, k_1, \ldots, k_{l-1}]$. 

For the rest cases, we do not use induction. For a design $\{m\}_n$, 
we take a sequence $r_0, r_1, \ldots, r_{t-1}$ and the realizing pair $(a, b)$ of the sequence as in Theorem \ref{Thm_FRC_Order} (2).  
Then we have $a=[2^n : m]$. 
We set the Euclidean algorithm: $a_i=a_{i+1}r_i+a_{i+2}$ \, $(i=0, 1, \ldots, t-1)$,
where $a_0=a, a_1=b, a_t=1$ and $a_{t+1}=0$.
We give a matrix presentation of these relations as follows:
\begin{equation} \label{Euclid's algorithm Matrix 1} \notag
\left(
\begin{matrix}
                          & -1                      &\!  r_0             &\! 1                    &\!               &\!\!              &\!\!\! {\Large \text{0}}         \\
                          &                          &\!  -1              &\! r_1                  &\! 1            &\!\!              &\!\!\!                                 \\
                          &                          &\!                   &\! \ddots             &\! \ddots     &\!\! \ddots    &\!\!\!                                 \\
                          &                          &\!                   &\!                       &\! -1           &\!\! r_{t-3}     &\!\!\! 1                               \\
                          &                          &\!                   &\!                       &\!              &\!\! -1          &\!\!\! r_{t-2}                         \\
                          & {\Large \text{0}}  &\!                   &\!                       &\!              &\!\!               &\!\!\! -1                              \\
\end{matrix}
\right) 
\left(
\begin{matrix}
                   a_0     \\
                   a_1     \\
                   a_2 \\
                   \vdots \\
                   \vdots \\
                   a_{t-1} 
\end{matrix}
\right)\!=\!
\left(
\begin{matrix}
                   0     \\
                  \vdots \\
                  \vdots \\
                   0     \\
                   -1     \\
                   -r_{t-1} 
\end{matrix}
\right). 
\end{equation}
By Cramer's rule and Lemma \ref{Lem_Special Matrix} (1), we have \\
\begin{equation} \notag
\begin{split}
a=a_0&=(-1)^t\left| \!\!\!\!\!
\begin{matrix}
                           & 0         &\!  r_0                    &\! 1                    &                   &\!\!             &\!\!\! {\Large \text{0}}         \\
                           & \vdots &\!  -1                     &\! r_1                  &1                 &\!\!             &\!\!\!          \\
                           & \vdots &\!                          &\! \ddots             & \ddots        &\!\! \ddots   &\!\!\!     \\
                           & 0         &\!                          &\!                       & -1              &\!\! r_{t-3}    &\!\!\! 1        \\
                           & -1       &\!                          &\!                       &                   &\!\! -1         &\!\!\! r_{t-2}  \\
                           & -r_{t-1}&\!{\Large \text{0}}    &\!                       &                   &\!\!             &\!\!\! -1        \\
\end{matrix}
\right| 
=\left| \!\!\!\!\!
\begin{matrix}
                           & r_0                  &\!   1            &\!                      &\!\!                 &\!\!\!              &\!\!\!\! {\Large \text{0}}         \\
                           & -1                   &\! r_1           &\! 1                    &\!\!                 &\!\!\!              &\!\!\!\!          \\
                           &                       &\! \ddots      &\! \ddots            &\!\! \ddots       &\!\!\!              &\!\!\!\!     \\
                           &                       &\!                &\! -1                  &\!\! r_{t-3}        &\!\!\! 1            &\!\!\!\!         \\
                           &                       &\!                &\!                      &\!\! -1             &\!\!\!r_{t-2}       &\!\!\!\! 1  \\
                           & {\Large \text{0}}&\!               &\!                      &\!\!                  &\!\!\!-1            &\!\!\!\!r_{t-1}        \\
\end{matrix}
\right| \\
&=[r_0, r_1, \ldots, r_{t-1}].
\end{split}
\end{equation}
(2) \, The case $k_{l-1}=1$. 
By taking $r_0, r_1, \ldots, r_{t-1}$ as in Theorem \ref{Thm_FRC_Order} (2) (ii), 
we have $t=l-1, r_i=k_i$ \, $(i=0, 1, \ldots, t-2=l-3)$ and $r_{t-1}=k_{l-2}+1$.
By Lemma \ref{Lem_Special Matrix} (4), we have
\[
a=[k_0, k_1, \ldots, k_{l-3}, k_{l-2}+1]=[k_0, k_1, \cdots, k_{l-3}, k_{l-2}, 1] \,\,\,(k_{l-1}=1).
\]
(3) \, The case $k_{l-1} \ge 2$. 
By taking $r_0, r_1, \ldots, r_{t-1}$ as in Theorem \ref{Thm_FRC_Order} (2) (i), 
we have $t=l, r_i=k_i$ \, $(i=0, 1, \ldots, t-1=l-1)$, and $a=[k_0, k_1, \ldots, k_{l-1}]$.
\end{proof}

\begin{Thm}
\label{Lem_4FRC_Matrix}
For any SDI $[2^n : m] \, (0 \le m \le 2^n-1)$, let $\{m\}_n$ be the design of $[2^n : m]$ such that
\[
\{m\}_n=
  \underbrace{1 \cdots 11}_{\tiny k_0}
  \underbrace{0 \cdots 00}_{\tiny k_1}
  \cdots
  \underbrace{0 \cdots 00}_{\tiny k_{l-2}}
  \underbrace{1 \cdots 11}_{\tiny k_{l-1}},
\]
where $l$ is odd, $k_0 \ge 0, k_i \ge 1\, (i=1, 2, \ldots, l-2)$ and $k_{l-1} \ge 0$. 
Then for $l \ge 3$, we have 
\[
\begin{split}
&(1)\,\,[2^n:m]=[k_0, k_1, \ldots,k_{l-1}]. \qquad\,\,\,\,\,\,(2)\,\, [2^n:2^n-m]=[k_1, k_2, \ldots,k_{l-1}]. \\
&(3)\,\,[2^n:m+1]=[k_0, k_1, \ldots,k_{l-2}]. \,\,\,\,\,\,\,\,(4)\,\,[2^n:2^n-(m+1)]=[k_1, k_2, \ldots,k_{l-2}].
\end{split}
\]
For $l=1$, we have $n=k_0, m=2^n-1. 
Then \, [2^n : m]=[2^n : 2^n-1]=n, [2^n : 2^n-m]=[2^n : 1]=1, [2^n : m+1]=[2^n : 2^n]=1$ and $[2^n : 2^n-(m+1)]=[2^n : 0]=0$.
\end{Thm}

\begin{proof} 
We only show the case $l \ge 3$. \\
(1) It is Theorem \ref{Thm_FRC_Matrix}. \\
(2) We have
\[
\{2^n-m\}_n=
\left\{
\begin{split}
&\underbrace{0 \cdots 00}_{\tiny k_0}
  \underbrace{1 \cdots 11}_{\tiny k_1}
  \cdots
  \underbrace{0 \cdots 00}_{\tiny k_{l-3}-1}
1
  \underbrace{0 \cdots 0}_{\tiny k_{l-2}}\,\,\, (k_{l-1}=0), \\
&\underbrace{0 \cdots 00}_{\tiny k_0}
  \underbrace{1 \cdots 11}_{\tiny k_1}
  \cdots
  \underbrace{1 \cdots 11}_{\tiny k_{l-2}}
  \underbrace{0 \cdots 0}_{\tiny k_{l-1}-1}1\,\,\, (k_{l-1} \ge 1). 
\end{split}
\right.
\]
By the same way as the proof of Theorem \ref{Thm_FRC_Matrix}, we have $[2^n:2^n-m]=[k_1, k_2, \cdots,k_{l-1}]$. \\
(3) We have
\[
\{m+1\}_n=
\left\{
\begin{split}
&\underbrace{1 \cdots 11}_{\tiny k_0}
  \underbrace{0 \cdots 00}_{\tiny k_1}
  \cdots
  \underbrace{1 \cdots 11}_{\tiny k_{l-3}}
  \underbrace{0 \cdots 0}_{\tiny k_{l-2}-1}1\,\,\, (k_{l-1}=0), \\
&\underbrace{1 \cdots 11}_{\tiny k_0}
  \underbrace{0 \cdots 00}_{\tiny k_1}
  \cdots
  \underbrace{0 \cdots 0}_{\tiny k_{l-2}-1}1
  \underbrace{0 \cdots 00}_{\tiny k_{l-1}}\,\,\, (k_{l-1} \ge 1). 
\end{split}
\right.
\]
By the same way as the proof of Theorem \ref{Thm_FRC_Matrix}, we have $[2^n:m+1]=[k_0, k_1, \ldots,k_{l-2}]$. \\
(4) By (2) and (3), we have the result.
\end{proof}

\begin{Def} \label{ASWQHOPM}
For a finite design 
$D=
  \underbrace{1 \cdots 11}_{\tiny k_0}
  \underbrace{0 \cdots 00}_{\tiny k_1}
  \cdots
  \underbrace{0 \cdots 00}_{\tiny k_{l-2}}
  \underbrace{1 \cdots 11}_{\tiny k_{l-1}},
$
where $l$ is odd, $k_0 \ge 0, k_i \ge 1\, (i=1, 2, \ldots, l-2)$ and $k_{l-1} \ge 0$,
the {\it inverse design} of $D$ is 
\[
D^*=
  \underbrace{1 \cdots 11}_{\tiny k_{l-1}}
  \underbrace{0 \cdots 00}_{\tiny k_{l-2}}
  \cdots
  \underbrace{0 \cdots 00}_{\tiny k_1}
  \underbrace{1 \cdots 11}_{\tiny k_0}.
\]
If $D=\{m\}_n$, then we denote by $D^*=\{m^*\}_n$.
We note that $(D^*)^*=D$.
\end{Def}

\begin{Thm}
\label{Thm_m^*}
For a finite design $\{m\}_n$ and its inverse design $\{m^*\}_n$, we have the following: 
\[
\begin{split}
&(1) \,\,\, [2^n:m]=[2^n:m^*]. \qquad \qquad \quad \,\,\,\,\,\, (2) \,\,\, [2^n:2^n-m]=[2^n:m^*+1]. \\
&(3) \,\,\, [2^n:m+1]=[2^n:2^n-m^*]. \qquad \,\, (4) \,\,\, [2^n:2^n-(m+1)]=[2^n:2^n-(m^*+1)].
\end{split}
\]
\end{Thm}

\begin{proof}
(1) By Lemma \ref{Lem_Special Matrix} (2) and Theorem \ref{Lem_4FRC_Matrix} (1), we have the result. \\
(2), (3) By Lemma \ref{Lem_Special Matrix} (2), Theorem \ref{Lem_4FRC_Matrix} (2) and (3), we have the result. \\
(4) By Lemma \ref{Lem_Special Matrix} (2) and Theorem \ref{Lem_4FRC_Matrix} (4), we have the result.
\end{proof}

\begin{Def}
\label{Def_Continued Fraction}
{\rm 
Let $k_0, k_1, k_2, \ldots, k_{l-1}$ be integers such that $l \ge 1, k_0 \ge 0$, and $ k_i \ge 1$ $(i=1, 2, \ldots, l-1)$.
We denote the continued fraction of the sequence by
\[
CF(k_0, k_1, \ldots, k_{l-2}, k_{l-1})
=k_0+\frac{1}{k_1} {\genfrac{}{}{0pt}{}{}{+ \dotsb \dotsb +}}
        \frac{1}{k_{l-2}} {\genfrac{}{}{0pt}{}{}{+}}
        \frac{1}{k_{l-1}} {\genfrac{}{}{0pt}{}{}{}}.
\]
That is, $CF(k_0, k_1, \ldots, k_{l-1})$ is obtained by computing the following recursive sequence
\[
CF(k_j, k_{j+1}, \ldots, k_{l-1})=k_j+\frac{1}{CF(k_{j+1}, \ldots, k_{l-1})} \,\,\,\,\, \text{from} \,\, j=l-2 \,\, \text{to} \,\, 0. 
\]
We note that
\[
\begin{split}
&CF(k_0, k_1, \ldots, k_{l-1})=1+CF(k_0-1, k_1, \ldots, k_{l-1})\,\,\,\,\,\,\,(k_0 \ge 1), \qquad \\
&CF(k_0, k_1, \ldots, k_{l-1})=CF(k_0, k_1, \ldots, k_{l-1}-1, 1)\,\,\,\,\,\,\,(k_{l-1} \ge 2). \qquad
\end{split}
\]
In computing $CF(k_0, k_1, \ldots, k_{l-1})$, if we allow computations such as $1/0=\infty$ and $1/\infty=0$, then we have 
$
CF(k_0, k_1, \ldots, k_{l-2}, 0)=CF(k_0, k_1, \ldots, k_{l-3})\,\,(l \ge 3).
$
If $k_{l-1}=0$ and $l=1$, then $CF(k_0, k_1, \ldots, k_{l-1})=CF(0)=0$.
If $k_{l-1}=0$ and $l=2$, then $CF(k_0, k_1, \ldots, k_{l-1})=CF(k_0, 0)=\infty$.
Hence we may assume that $k_{l-1} \ge 0$. \\
}
\end{Def}

The following is a relationship between the continuant and the continued fractions.

\begin{Lem}
\label{Thm_FRC_Product1}
Let $k_0, k_1, k_2, \ldots, k_{l-1}$ be integers such that $l \ge 1, k_0 \ge 0, k_i \ge 1$ $(i=1, 2, \ldots, l-2)$ and $k_{l-1} \ge 0$.
Then we have the following:
\[
\begin{split}
(1) \,\, &{\rm(i)} \,\, \text{If} \,\, k_{\tiny l-1} \ge 1, \text{then we have} \,\,
[k_0, k_1, \ldots, k_{l-1}]=\prod_{j=0}^{l-1}CF(k_j, k_{j+1}, \ldots, k_{l-1}). \\
&{\rm (ii)} \,\, \text{If} \,\, k_{\tiny l-1}=0, \text{then we have} \,\,
[k_0, k_1, \ldots, k_{l-1}]=\prod_{j=0}^{l-3}CF(k_j, k_{j+1}, \ldots, k_{l-3}) \,\,\,\,\,(l \ge 3). \\
(2) \,\, &{\rm (i)} \,\, \text{If} \,\, k_0 \ge 1, \text{then we have} \,\,
[k_0, k_1, \ldots, k_{l-1}]=\prod_{j=0}^{l-1}CF(k_j, k_{j-1}, \ldots, k_0). \\
&{\rm (ii)} \,\, \text{If} \,\, k_0=0, \text{then we have} \,\,
[k_0, k_1, \ldots, k_{l-1}]=\prod_{j=2}^{l-1}CF(k_j, k_{j-1}, \ldots, k_2) \,\,\,\,\,(l \ge 3). 
\end{split}
\]
\end{Lem}

\begin{proof}
(1) (i) \, By Lemma \ref{Lem_Special Matrix} (1) and elementary transformations of matrices, we have
\[
\begin{split}
&[k_0, k_1, \cdots, k_{l-1}] 
=\left|
\begin{matrix}
              k_0 & 1                           &        &                 &                    & {\Large \text{0}} \\
              -1  &k_1                          &1        &             &                        &   \\
                   & \ddots                                  &\ddots     & \ddots               &                     &   \\
                  &                                   &\!\!\!\!\!-1        &\!\!\!\!\!k_{l-3}                &\!\!\!\!\! 1                    &   \\
                        &                             &       &\!\!\!\!\!\!\!\!\!\!-1         &\!\!\!\!\!k_{l-2}                      &\!\!\!\!\! 1 \\
{\Large \text{0}} &                             &        &         &\!\!\!\!\!\!\!\!\!\!-1                          &\!\!\!\!\! k_{l-1} 
\end{matrix}
\right|
=\left|
\begin{matrix}
              k_0 & 1                           &        &                 &                    & {\Large \text{0}} \\
              -1  &k_1                          &1        &             &                        &   \\
                   & \ddots                                  &\ddots     & \ddots               &                     &   \\
                  &                                   &\!\!\!\!\!-1        &\!\!\!\!\!k_{l-3}                &\!\!\!\!\!\!\!\! 1                    &   \\
                        &                             &       &\!\!\!\!\!\!\!\!-1         &\!\!\!\!\!\!\!\!k_{l-2}\!+\!\cfrac{\!1}{k_{l-1}}                      &\!\!\!\!\! 1 \\
{\Large \text{0}} &                             &        &         &\!\!\!\!\!\!\!\!\!\!0                          &\!\!\!\!\! k_{l-1} 
\end{matrix}
\right| \\
&\,\,=\left|
\begin{matrix}
              k_0 & 1                           &        &                 &                    & {\Large \text{0}} \\
              -1  &k_1                          &1        &             &                        &   \\
                   & \ddots                                  &\ddots     &\!\!\!\!\!\!\!\!\!\!\!\!\!\!\!\!\!\!\!\!\!\!\!\!\!\!\! \ddots               &                     &   \\
                  &                                   &\!\!\!\!\!-1        &\!\!\!\!\! k_{l-3}\!+\!\cfrac{\!1}{k_{l-2}\!+\!\cfrac{\!1}{k_{l-1}}}     &\!\!\!\!\!\!\!\!\! 1                    &   \\
                        &                             &       &\,\,\,\,\,\,\,\,0         &\!\!\!\!\!\!\!\!k_{l-2}\!+\!\cfrac{\!1}{k_{l-1}}          &\!\! 1 \\
{\Large \text{0}} &                             &        &         &\!\!\!\!\!\!\!\!\!\!0                          &\!\!\!\!\! k_{l-1} 
\end{matrix}
\right| 
=\cdots \cdots=\prod_{j=0}^{l-1}CF(k_j, k_{j+1}, \ldots, k_{l-1}).
\end{split}
\]
(ii) \, Since $[k_0, k_1, \ldots, k_{l-2},0]=[k_0, k_1, \ldots, k_{l-3}]$ by Lemma \ref{Lem_Special Matrix} (6), we have the result. \\
(2) \, By (1) and Lemma \ref{Lem_Special Matrix} (2), we have the result.
\end{proof}

The following is a relationship between SDIs and the continued fractions.

\begin{Thm}
\label{Lem_FRC_Product2}
Let
$
\{m\}_n=
           \underbrace{1 \cdots 11}_{\tiny k_0}
           \underbrace{0 \cdots 00}_{\tiny k_1}
           \cdots
           \underbrace{0 \cdots 00}_{\tiny k_{l-2}}
           \underbrace{1 \cdots 11}_{\tiny k_{l-1}}
$
be a finite design, where $l$ is odd, $k_0 \ge 0, k_i \ge 1\,\,(i=1, 2, \ldots, l-2)$ and $k_{l-1} \ge 0$.
Then we have the following:
\[
\begin{split}
(1) \,\, &{\rm (i)} \,\, \text{If} \,\, k_{\tiny l-1} \ge 1, \text{then we have} \,\,
[2^n : m]=\prod_{j=0}^{l-1}CF(k_j, k_{j+1}, \ldots, k_{l-1}). \\
&{\rm (ii)} \,\, \text{If} \,\, k_{\tiny l-1}=0, \text{then we have} \,\,
[2^n : m]=\prod_{j=0}^{l-3}CF(k_j, k_{j+1}, \ldots, k_{l-3}) \,\,\,\,\,(l \ge 3). \qquad \qquad \\
(2) \,\, &{\rm (i)} \,\, \text{If} \,\, k_0 \ge 1, \text{then we have} \,\,
[2^n : m]=\prod_{j=0}^{l-1}CF(k_j, k_{j-1}, \ldots, k_0). \qquad \qquad \\
&{\rm (ii)} \,\, \text{If} \,\, k_0=0, \text{then we have} \,\,
[2^n : m]=\prod_{j=2}^{l-1}CF(k_j, k_{j-1}, \ldots, k_2) \,\,\,\,\,(l \ge 3). \qquad \\
(3) \,\, &{\rm (i)} \,\, \frac{[2^n : m]}{[2^n : 2^n-m]}=CF(k_0, k_1, \ldots, k_{l-1}). \quad
{\rm (ii)} \,\, \frac{[2^n : m]}{[2^n : m+1]}=CF(k_{l-1}, k_{l-2}, \cdots, k_0). 
\end{split}
\]
\end{Thm}

\begin{proof}
(1) (i) By Theorem \ref{Thm_FRC_Matrix} and Lemma \ref{Thm_FRC_Product1} (1) (i), we have the result. \\
(ii) By Theorem \ref{Thm_FRC_Matrix} and Lemma \ref{Thm_FRC_Product1} (1) (ii), we have the result. \\
(2) (i) By Theorem \ref{Thm_FRC_Matrix} and Lemma \ref{Thm_FRC_Product1} (2) (i), we have the result. \\
(ii) By Theorem \ref{Thm_FRC_Matrix} and Lemma \ref{Thm_FRC_Product1} (2) (ii), we have the result. \\
(3) (i) By (1) and Theorem \ref{Lem_4FRC_Matrix} (2), we have the result. \\
(ii) By (2) and Theorem \ref{Lem_4FRC_Matrix} (3), we have the result. 
\end{proof}

\begin{Rem} \label{AYJWDSAS}
{\rm
(1) By Definition \ref{Def_SDA} (2), in the left-hand side of the equation in Theorem \ref{Lem_FRC_Product2} (3) (i), we may assume that $m$ is odd. \\
(2) In Theorem \ref{Lem_FRC_Product2} (3), we may assume $k_0=0$ and/or $k_{l-1}=0$. 
}
\end{Rem}

\begin{Cor} \label{BESHSWQA}
Under the setting in Theorem \ref{Lem_FRC_Product2}, we have the following: \\
$(1)$ \, $k_0 \ge 1$ if and only if \, $[2^n : m] \ge [2^n : 2^n-m]$. \\
$(2)$ \, $k_{l-1} \ge 1$ if and only if \, $[2^n : m] \ge [2^n : m+1]$. 
\end{Cor}

\begin{proof}
(1) \, Since $CF(k_0, k_1, \ldots, k_{l-1}) \ge 1$, we have the result by Theorem \ref{Lem_FRC_Product2} (3) (i). \\
(2) \, Since $CF(k_{l-1}, \ldots, k_1, k_0) \ge 1$, we have the result by Theorem \ref{Lem_FRC_Product2} (3) (ii). \\
\end{proof}

The following statement is the basis to define the {\it assembly function} in Section 6.

\begin{Cor}
\label{Thm_RRC_Continued Fraction}
Let $a$ and $b$ be positive coprime integers. 
Then there exists a unique pair of two positive integers $(m, n)$ such that $m$ is odd, $1 \le m \le 2^n-1$ and
\[
\frac{a}{\,b\,}=\frac{[2^n:m]}{[2^n:2^n-m]}.
\]
\end{Cor}

\begin{proof}
There exists a unique sequence of integers $k_0, k_1, \ldots, k_{l-1}$ such that $l$ is odd, $k_0 \ge 0, k_i \ge 1 \, (i=1, 2, \ldots, l-1)$ and
\[
\frac{a}{\,b\,}=CF(k_0, k_1, \ldots, k_{l-1}).
\]
The sequence uniquely determines $\{m\}_n$ and a pair $(m, n)$ such that $m$ is odd, $1 \le m \le 2^n-1$ as in Theorem \ref{Lem_FRC_Product2}. 
\end{proof}

We can also prove Corollary \ref{Thm_RRC_Continued Fraction} by Theorem \ref{Thm_FRC_Order} directly.

\section{Stern's diatomic matrix}

In this section, we define the {\it Stern's diatomic matrix}.

\begin{Def}
\label{Def_Modular Matrix PPP}
{\rm For two non-negative integers $m$ and $n$ with $0 \le m \le 2^n-1$, we define a {\it Stern's diatomic matrix} (SDM, for short) by
\[
U (2^n : m)
=\left(
\begin{matrix}
[2^n:m+1]& [2^n:m] \\
[2^n : 2^n-(m+1)]& [2^n : 2^n-m]
\end{matrix}
\right).
\]
We also denote it by
$
U (\{m\}_n)$,
where $\{m\}_n$ is the design of $[2^n : m]$.
We call $\{m\}_n$ the {\it design} of $U (2^n : m)$.
Then by Theorem \ref{Thm_det=1}, the determinant $|U (2^n : m)|=1$, and for the empty design $\varepsilon=\{0\}_0$, we have
$
U(\varepsilon)
=
\left(
\begin{matrix}
1& 0 \\
0& 1 
\end{matrix}
\right).
$
}
\end{Def}

\begin{Lem} \label{HFNJWTTN}
For an integer $n \ge 0$, we have 
\[
U(\{0\}_n)=U( \underbrace{0 \cdots 00}_{\tiny n})=
\left(
\begin{matrix}
1& 0 \\
n& 1 
\end{matrix}
\right) 
\,\,\,\,\, \text{and} \,\,\,\,\, 
U(\{2^n-1\}_n)=U( \underbrace{1 \cdots 11}_{\tiny n})=
\left(
\begin{matrix}
1& n \\
0& 1 
\end{matrix}
\right).
\]
\end{Lem}

\begin{proof}
By Lemma \ref{Lem_FRC-formula1} (1), we have the result. 
\end{proof}

\begin{Lem}
\label{Lem_Modular Matrix}
Let $a, b, c$ and $d$ be non-negative integers such that $ad-b c=1$ and $(a, b, c, d) \neq (1, 0, 0, 1)$. Then we have the following: \\
$(1)$ \, $a \ge 1$ and $d \ge 1$. \\
$(2)$ \, If $b=0$ or $c=0$, then $a=d=1$. \\
$(3)$ \, {\rm  (i)} \, If $b \neq 0$, then $a/b >c/d$. \\
{\rm  (ii)} \, If $c \neq 0$, then $a/c >b/d$. \\
$(4)$ \, $a \ge b, c \ge d$ or $a \le b, c \le d$. \\
\end{Lem}

\begin{proof}
(1) Since $ad=bc+1 \ge 1$, we have the result. \\
(2) If $b=0$ or $c=0$, then $ad=1$ and $a=d=1$. \\
(3) (i) By (1), $d > 0$. Suppose $b > 0$. Since $\cfrac{a}{\,b\,}-\cfrac{c}{\,d\,}=\cfrac{ad-bc}{bd}=\cfrac{1}{bd} > 0$, we have the result. \\
(ii) By the similar way as (i), we have the result. \\
(4) Suppose $b=0$. Then by (2) and the assumption $(a, b, c, d) \neq (1, 0, 0, 1)$, we have $a=d=1$ and $c \ge 1$.
That is, $a \ge b, c \ge d$ hold. 

Suppose $b > 0$. If $1 \ge a/b > c/d$ or $ a/b > c/d \ge 1$, then we have the result.
Hence we suppose $a/b \ge 1 \ge c/d$. Then by (3) (i), we have $a \ge b, d \ge c+1$ or $a \ge b+1, d \ge c$. \\
Suppose $a \ge b, d \ge c+1$. Then $1=ad-bc \ge a(c+1)-ac=a >0$.
Hence we have $a=b=1$ and $d=c+1$. That is, $a \le b, c \le d$ hold. The case $a \ge b+1, d \ge c$ can be shown in the same way. 
\end{proof}

The following theorem states that the set of finite designs has one to one correspondence with the set of unimodular matrices with non-negative entries.
We call it the {\it Matrix Representation Theorem}.

\begin{Thm} {\rm (Matrix Representation Theorem)} \\
\label{Thm_Modular Matrix_Depth Order}
For four non-negative integers $a, b, c$ and $d$ with $ad-bc=1$, 
there exists a unique pair of two non-negative integers $(m, n)$ such that $0 \le m \le 2^n-1$, and 
\[
U(2^n : m)=U(\{m\}_n)=
\left(
\begin{matrix}
a &b\\
c &d
\end{matrix}
\right).
\]
\end{Thm}

\begin{proof}
By Lemma \ref{Lem_Modular Matrix} (1), we have $d > 0$. \\
(1) Suppose $b > 0$. 
Then there exists a unique pair of two non-negative integers $(x_0, y_0)$ such that $dx_0-by_0=1, 0 \le x_0 \le b$ and $0 \le y_0 \le d$.
Since $b$ and $d$ are coprime, any non-negative integral solution of $dx-by=1$ is of the form $(x, y)=(x_0+sb, y_0+sd)$,
where $s$ is a non-negative integer. 
In particular, there exists a unique non-negative integer $k$ such that $(a, c)=(x_0+kb, y_0+kd)$.
By Theorem \ref{Thm_FRC_Order} (1), there exists a unique pair of two non-negative integers $(m_0, n_0)$
such that $m_0$ is odd, $1 \le m_0 \le 2^{n_0}-1, [2^{n_0} : m_0]=b$ and $[2^{n_0} : 2^{n_0}-m_0]=d$.
Here, we set
$
\{m_0\}_{n_0}=
  \underbrace{1 \cdots 11}_{\tiny k_0}
  \underbrace{0 \cdots 00}_{\tiny k_1}
  \cdots
  \underbrace{0 \cdots 00}_{\tiny k_{l-2}}
  \underbrace{1 \cdots 11}_{\tiny k_{l-1}},
$
where $l$ is odd and $k_{l-1} \ge 1$.
By Theorem \ref{Thm_det=1} and Corollary \ref{BESHSWQA} (2), 
we have $(x_0, y_0)=([2^{n_0} : m_0+1], [2^{n_0} : 2^{n_0}-(m_0+1)])$.
Hence, by Lemma \ref{Lem_FRC-formula1} (1) and (2), we have
\[
\begin{split}
&a=x_0+kb=[2^{n_0} : m_0+1]+k[2^{n_0} : m_0]=[2^{n_0+k} : 2^km_0+1], \\
&c=y_0+kd=[2^{n_0} : 2^{n_0}-(m_0+1)]+k[2^{n_0} : 2^{n_0}-m_0]=[2^{n_0+k} : 2^{n_0+k}-(2^km_0+1)], \\
&b=[2^{n_0} : m_0]=[2^{n_0+k} : 2^km_0], \\
&d=[2^{n_0} : 2^{n_0}-m_0]=[2^{n_0+k} : 2^{n_0+k}-2^km_0].
\end{split}
\]
Therefore, by setting $n=n_0+k$ and $m=2^km_0$, we have
\[
U(2^n : m)=
\left(
\begin{matrix}
a &b\\
c &d
\end{matrix}
\right).
\]
(2) If $b=0$, by Lemma \ref{Lem_Modular Matrix} (2), we have $a=d=1$, and we can take $(x_0, y_0)=(1, 0), k=c$ and $(x, y)=(1, s)$. 
Then we can take $(m_0, n_0)=(0, 0)$ (i.e. $[2^{n_0} : m_0]=0$ and $[2^{n_0} : 2^{n_0}-m_0]=1$) uniquely (i.e. minimal $n_0$).
Also in this case, we can take $(x_0, y_0)=([2^{n_0} : m_0+1], [2^{n_0} : 2^{n_0}-(m_0+1)])=(1, 0)$. 
Hence, in the same way as (1), we have  
\[
U(2^{n_0+k} : 2^km_0)=U(2^c : 0)=
\left(
\begin{matrix}
[2^c : 1]        &[2^c : 0] \\
[2^c : 2^c-1]  &[2^c : 2^c]
\end{matrix}
\right)
=\left(
\begin{matrix}
1 & 0\\
c & 1
\end{matrix}
\right).
\]
\end{proof}

\begin{Exm} \label{ZSWJYTS}
{\rm We observe the designs of the following matrices. 
\[
\left(
\begin{matrix}
5 & 7 \\
2 & 3
\end{matrix}
\right)
=
U(11001),
\quad
\left(
\begin{matrix}
8 & 3 \\
5 & 2
\end{matrix}
\right)
=
U(10100),
\]
where 11001 is the Euclidean design of 7 generated by 3
and 10101 is the Euclidean design of 8 generated by 5.
Further from Theorem \ref{Thm_Modular Matrix_Depth Order}, we have
\[
\left(
\begin{matrix}
1 & b \\
0 & 1
\end{matrix}
\right)
=U(\underbrace{1 \cdots 11}_{b}), \quad
\left(
\begin{matrix}
1 & 0 \\
c & 1
\end{matrix}
\right)
=U(\underbrace{0 \cdots 00}_{c}), \quad
\left(
\begin{matrix}
1 & 0 \\
0 & 1
\end{matrix}
\right)
=U(\varepsilon).
\]}
\end{Exm}

\begin{Thm}
Suppose
$
U(\{m\}_n)
=U(
     \underbrace{1 \cdots 11}_{\tiny {k_0}}
     \underbrace{0 \cdots 00}_{\tiny {k_1}}
     \cdots
     \underbrace{0 \cdots 00}_{\tiny {k_{l-2}}}
     \underbrace{1 \cdots 11}_{\tiny {k_{l-1}}}
)
=
\left(
\begin{matrix}
 a  &  b \\
 c  &  d
\end{matrix}
\right),
$
where 
$
\{m\}_n=
  \underbrace{1 \cdots 11}_{\tiny {k_0}}
  \underbrace{0 \cdots 00}_{\tiny {k_1}}
  \cdots
  \underbrace{0 \cdots 00}_{\tiny {k_{l-2}}}
  \underbrace{1 \cdots 11}_{\tiny {k_{l-1}}}
$
is the design of the matrix. 
Then we have the following: 
\[
\begin{split}
&{\rm (1)} \,\, 
U(\{m^*\}_n)
=U(\underbrace{1 \cdots 11}_{\tiny {k_{l-1}}}
     \underbrace{0 \cdots 00}_{\tiny {k_{l-2}}}
     \cdots
     \underbrace{0 \cdots 00}_{\tiny {k_1}}
     \underbrace{1 \cdots 11}_{\tiny {k_0}})
=\left(
\begin{matrix}
 d  &  b \\
 c  &  a
\end{matrix}
\right).  \\
&{\rm (2)} \,\,  
U(\{2^n-(m^*+1)\}_n)
=U(\underbrace{0 \cdots 00}_{\tiny {k_{l-1}}}
     \underbrace{1 \cdots 11}_{\tiny {k_{l-2}}}
     \cdots
     \underbrace{1 \cdots 11}_{\tiny {k_1}}
     \underbrace{0 \cdots 00}_{\tiny {k_0}})
=\left(
\begin{matrix}
 a  &  c \\
 b  &  d
\end{matrix}
\right). \\
&{\rm (3)} \,\, 
U(\{2^n-(m+1)\}_n)
=U(\underbrace{0 \cdots 00}_{\tiny {k_0}}
     \underbrace{1 \cdots 11}_{\tiny {k_1}}
     \cdots
     \underbrace{1 \cdots 11}_{\tiny {k_{l-2}}}
     \underbrace{0 \cdots 00}_{\tiny {k_{l-1}}})
=\left(
\begin{matrix}
 d  &  c \\
 b  &  a
\end{matrix}
\right). \qquad
\end{split}
\]
\end{Thm}

\begin{proof}
By Theorem \ref{Thm_m^*}, we have the result.
\end{proof}

\section{Definition of the assembly function}

This section is the main part of the paper.
We define the {\it assembly function} and show some fundamental properties.

\begin{Def}
\label{Def_Ichio Function}
{\rm 
Let $S$ be the set of rational numbers of the form $m/2^n$, 
where $m$ and $n$ are two non-negative integers with $0 \le m \le 2^n-1$, and $\mathbb{R}_{ \ge 0}$ the set of non-negative real numbers.
Then we define a map $\mathcal{A} : S \longrightarrow \mathbb{R}_{ \ge 0}$ by
\[
\mathcal{A}\left(\frac{m}{2^n} \right)=\frac{[2^n:m]}{[2^n:2^n-m]}.
\]
Since $[2^{n+1} : 2m]=[2^n : m]$ in general 
(cf. Definition \ref{Def_SDA} (2)), $\mathcal{A}$ is well-defined. 
We call $\mathcal{A}$ the {\it rational assembly function}.
}
\end{Def}

\begin{Thm} \label{KIJLMMKU}
The rational assembly function extends uniquely to a continuous map $[0, 1) \longrightarrow [0, \infty)$,
and the map is strictly increasing and bijective.
\end{Thm}

\begin{proof}
We assume the setting in Definition \ref{Def_Ichio Function}. 

Firstly, we show that $\mathcal{A} : S \longrightarrow \mathbb{R}_{ \ge 0}$ is strictly increasing. 
Suppose that $0 \le m < m+1 \le 2^n-1$. Since
\[
\begin{split}
\mathcal{A}\left(\frac{m+1}{2^n} \right)
-\mathcal{A}\left(\frac{m}{2^n} \right)&=\frac{[2^n : m+1]}{[2^n : 2^n-(m+1)]}-\frac{[2^n : m]}{[2^n : 2^n-m]} 
\end{split}
\]
\[
\begin{split}
&=\frac{[2^n : m+1][2^n : 2^n-m]-[2^n : m][2^n : 2^n-(m+1)]}{[2^n : 2^n-m][2^n : 2^n-(m+1)]} \\[5pt]
&=\frac{1}{[2^n : 2^n-m][2^n : 2^n-(m+1)]} > 0
\end{split}
\]
by Theorem \ref{Thm_det=1}, $\mathcal{A}$ is strictly increasing. 

By Theorem \ref{Thm_FRC_Order} (Corollary \ref{Thm_RRC_Continued Fraction}), $\mathcal{A}$ is surjective.
Since $S$ is dense in $[0, 1)$ and $\mathbb{Q}_{ \ge 0}$ is dense in $\mathbb{R}_{ \ge 0}=[0, \infty)$, we have the result.
\end{proof}

\begin{Def}
{\rm We call the extended map of $\mathcal{A}$ the {\it assembly function}, 
and use the same notation $\mathcal{A}$.
Further we define $\mathcal{A}(1)=\infty$.
}
\end{Def}

The assembly function has a characteristic graph shown in Figure 6-1.


\begin{figure}[htbp]
\begin{center}
\includegraphics[scale=0.3]{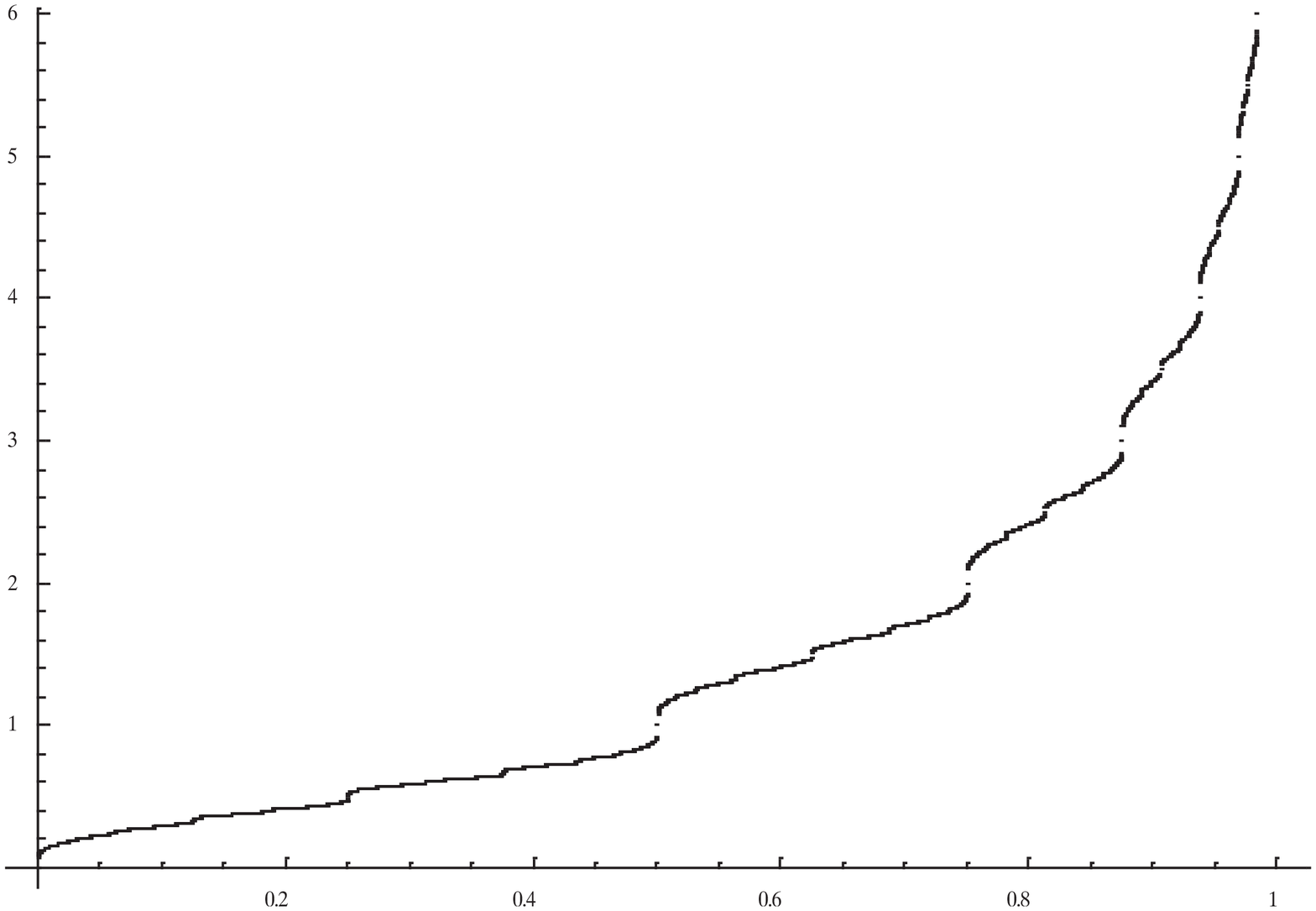} 

\small{Fig. 6-1}
\end{center}
\end{figure}

\begin{Lem} \label{SKWIDDGJGE} 
For $m$ and $n$ are two non-negative integers with $0 \le m \le 2^n-1$, 
and $m$ is odd,
suppose the design of $[2^n:m]$ is 
$\underbrace{1 \cdots 1}_{\tiny k_0}\underbrace{0 \cdots 0}_{\tiny k_1}
\cdots \underbrace{1 \cdots 1}_{\tiny k_{l-1}}\ 
(k_0 \ge 0, k_i \ge 1\ (1 \le i \le l-1))$.
Then we have
$$\mathcal{A}\left(
\frac{m}{2^n}
\right)
=CF(k_0, k_1, \ldots, k_{l-1}).$$

\end{Lem}

\begin{proof}
By Definition \ref{Lem_FRC_Product2} (3), we have the result.
\end{proof}

\begin{Lem} \label{EGHSRTWED}
Let $\omega=\mathcal{A}(\theta) \,\,\, (\theta \in [0,1])$ be the assembly function. Then we have 
\[
\begin{split}
&(1) \,\, \mathcal{A}(1/2)=1, \,\,\,\, 0 \le \mathcal{A}(\theta) <1 \,\,\, (0 \le \theta <1/2), \,\,\,\, \mathcal{A}(\theta) >1 \,\,\, (\theta > 1/2). \\
&(2) \,\, \mathcal{A}(1-\theta)=\frac{1}{\mathcal{A}(\theta)} \,\,\, (0 \le \theta \le 1). \qquad \qquad \qquad \qquad \qquad \qquad \qquad \qquad \qquad \qquad
\end{split}
\]
\end{Lem}

\begin{proof}
(1) By Definition \ref{Def_Ichio Function}, we have
$
\mathcal{A}(1/2)=\cfrac{[2:1]}{[2:1]}=1.
$
By Theorem \ref{KIJLMMKU}, we have the result. \\
(2) For $\theta=m/2^n$, where $m$ and $n$ are positive integers with $1 \le m \le 2^n-1$, 
by Definition \ref{Def_Ichio Function}, we have
\[
\mathcal{A}(1-\theta)=\mathcal{A}\left(\frac{2^n-m}{2^n}\right)=\frac{[2^n : 2^n-m]}{[2^n : m]}=\frac{1}{\mathcal{A}(\theta)}.
\]
By Theorem \ref{KIJLMMKU} (continuity of $\mathcal{A}(\theta)$), we have the result.
\end{proof}

\begin{Def} \label{AELMUQSST}
{\rm
For a finite design 
$
\{m\}_n=
  \underbrace{1 \cdots 11}_{\tiny k_0}
  \underbrace{0 \cdots 00}_{\tiny k_1}
  \cdots
  \underbrace{0 \cdots 00}_{\tiny k_{l-2}}
  \underbrace{1 \cdots 11}_{\tiny k_{l-1}},
$
where $l$ is odd, $k_0 \ge 0, k_i \ge 1\, (i=1, 2, \ldots, l-2), \, k_{l-1} \ge 0$ and $n=\sum_{i=0}^{l-1}k_i$,
we define a binary decimal by
\[
\frac{\{m\}_n}{2^n}=\frac{m}{2^n}=
0.\underbrace{1 \cdots 11}_{\tiny k_0}
  \underbrace{0 \cdots 00}_{\tiny k_1}
  \cdots
  \underbrace{0 \cdots 00}_{\tiny k_{l-2}}
  \underbrace{1 \cdots 11}_{\tiny k_{l-1}},
\]
and call it a {\it design decimal} of $\{m\}_n$. 

For a terminal design $\{2^n\}_n=1\underbrace{0 \cdots 00}_{\tiny n}{}_{\tau}$ with $n \ge 0$, 
we define a binary decimal by
\[
\frac{\{2^n\}_n}{2^n}=\frac{2^n}{2^n}=1,
\]
and call it a {\it design decimal} of $\{2^n\}_n$. 

For an integer sequence $\{k_j\}_{j=0}^{\infty}$ such that $k_0 \ge 0$ and $k_j \ge 1 \, (j=1, 2, \ldots)$,
we define an {\it infinite design} as an infinite sequence of finite designs:
\[
  \underbrace{1 \cdots 11}_{\tiny k_0}
  \underbrace{0 \cdots 00}_{\tiny k_1}
  \cdots
  \underbrace{0 \cdots 00}_{\tiny k_{l-2}}
  \underbrace{1 \cdots 11}_{\tiny k_{l-1}} \cdots \cdots=
\Bigg \{
  \underbrace{1 \cdots 11}_{\tiny k_0}
  \underbrace{0 \cdots 00}_{\tiny k_1}
  \cdots
  \underbrace{0 \cdots 00}_{\tiny k_{2j-1}}
  \underbrace{1 \cdots 11}_{\tiny k_{2j}}
\Bigg \}_{j=0}^{\infty}
\]
and call its $j$-th term the {\it $j$-th segment} of the infinite design.
Then we define a binary decimal by 
\[
\lim_{j \to \infty}\frac
 {\overbrace{1 \cdots 11}^{\tiny k_0}
  \overbrace{0 \cdots 00}^{\tiny k_1}
  \cdots
  \overbrace{0 \cdots 00}^{\tiny k_{2j-1}}
  \overbrace{1 \cdots 11}^{\tiny k_{2j}}}
{2^{\sum_{i=0}^{2j}k_{i}}}=
0.
  \underbrace{1 \cdots 11}_{\tiny k_0}
  \underbrace{0 \cdots 00}_{\tiny k_1}
  \cdots
  \underbrace{0 \cdots 00}_{\tiny k_{l-2}}
  \underbrace{1 \cdots 11}_{\tiny k_{l-1}} \cdots \cdots,
\]
and call it a {\color{black}{\it design decimal}} of 
$
  \underbrace{1 \cdots 11}_{\tiny k_0}
  \underbrace{0 \cdots 00}_{\tiny k_1}
  \cdots
  \underbrace{0 \cdots 00}_{\tiny k_{l-2}}
  \underbrace{1 \cdots 11}_{\tiny k_{l-1}} \cdots \cdots.
$ 

For a design $D$, we denote the design decimal of $D$ by $\theta_D$.
We note that for any real number $\theta$ with $0 \le \theta \le 1$, 
there exists a design $D$ such that $\theta=\theta_D$. 
If $\theta$ is of the form $m/2^n$, a reduced design $D$ is uniquely determined,
and if $\theta$ is not of the form $m/2^n$, an infinite design $D$ is uniquely determined.
Then we call $D$ the {\it design} of $\theta$.
}
\end{Def}

\begin{Def}
\label{Def_Design of Matrix}
{\rm 
Let $\varGamma$ be the set of elements in $SL(2 : \mathbb{Z})$ (the special linear group over $\mathbb{Z}$) with non-negative entries:
\[
\varGamma=\Bigg \{
\left(
\begin{matrix}
a & b \\
c & d
\end{matrix}
\right)
\in SL(2 \, ; \, \mathbb{Z}) \, \Bigg | \, a, b, c, d \in \mathbb{Z}_{\ge 0}
\Bigg \}.
\]
Then $A \in \varGamma$ acts on the set $\varDelta= \big \{ (p, q) \in (\mathbb{Z}_{\ge 0})^2 \,\, | \,\, gcd(p, q)=1 \big \}$ \,
and \, $\mathbb{Q}_{\ge 0} \cup \{ \infty \} \,\, (\infty=1/0)$ as follows:
\[
\qquad A
\left(
\begin{matrix}
p \\
q 
\end{matrix}
\right)
=
\left(
\begin{matrix}
ap+bq \\
cp+dq 
\end{matrix}
\right) \,\,\, \big ((p, q) \in \varDelta \big )
\]
and
\[
Ax=
\left\{
\begin{split}
& \frac{ax+b}{cx+d} \,\,\,(x \in \mathbb{Q}_{\ge 0}), \\[8pt]
& \frac{a}{\,c\,} \qquad \,\,\, (x=\infty).
\end{split}
\right.
\]
Then the action is called a {\it linear fractional transformation}.
Let $\mathcal{I}(S)$ be the set of injections from $S$ to $S$.
The action above defines a map
\[
\varPhi \, : \, \varGamma \longrightarrow \mathcal{I}(\varDelta) \,\,\, \text{and} \,\,\, \mathcal{I}(\mathbb{Q}_{\ge 0} \cup \{ \infty \})
\]
naturally, and the map is a monoid homomorphism (i.e. $\varPhi(BA)=\varPhi(B) \circ \varPhi(A) \,(A, B \in \varGamma)$ and $\varPhi(E)=id.$, 
where $E$ is the unit matrix and {\it id.} is the identity map). 
We note that there is a natural bijection between $\varDelta$ and $\mathbb{Q}_{\ge 0} \cup \{ \infty \}$, and $\varGamma$ is a free monoid generated by 
\[
\left(
\begin{matrix}
1 & 1 \\
0 & 1
\end{matrix}
\right)
\,\,\, \text{and} \,\,\,
\left(
\begin{matrix}
1 & 0 \\
1 & 1
\end{matrix}
\right).
\]
}
\end{Def}

Let $k_0, k_1, \ldots, k_{l-1}$ be integers such that $l$ is odd, $k_0 \ge 0, k_i \ge 1 \, (i=1, 2, \ldots, l-2)$ and $k_{l-1} \ge 0$.
We regard $CF(k_0, k_1, \ldots, k_{l-2}, k_{l-1}+x)$ as a function of $x$, where $x$ is a real variable with $x \ge 0$.
Then it is a continuous function.

\begin{Lem} \label{DDUGEADXC}
For  a finite design
$
\{m\}_n=
  \underbrace{1 \cdots 11}_{\tiny k_0}
  \underbrace{0 \cdots 00}_{\tiny k_1}
  \cdots
  \underbrace{0 \cdots 00}_{\tiny k_{l-2}}
  \underbrace{1 \cdots 11}_{\tiny k_{l-1}}
$\,,
we have
\[
CF(k_0, k_1, \ldots, k_{l-2}, k_{l-1}+x)=\frac{[2^n : m+1]x+[2^n : m]}{[2^n : 2^n-(m+1)]x+[2^n : 2^n-m]}
=U(2^n : m)x.
\]
\end{Lem}

\begin{proof}
Suppose $l=1$. Then we have $k_0=n \, (\ge 0), m=2^n-1$ and $CF(k_0+x)=x+n$.
By Lemma \ref{Lem_FRC-formula1} (1), the equation holds. 

Suppose $l \ge 3$. We replace every $k_i \, (i=0, 1, \ldots, l-1)$ with a real variable $x_i$, 
where $x_0 \ge 0, x_i \ge 1 \, (i=1, 2, \ldots, l-2)$ and $x_{l-1} \ge 0$.
Then we can generalize Definition \ref{Def_Continued Fraction} and Lemma \ref{Thm_FRC_Product1} for variable cases.
By generalized Lemma \ref{Thm_FRC_Product1} (1), we have
\[
CF(x_0, x_1, \ldots, x_{l-2}, x_{l-1}+x)=\frac{[x_0, x_1, \ldots, x_{l-2}, x_{l-1}+x]}{[x_1, x_2, \ldots, x_{l-2}, x_{l-1}+x]}
\]
(The case $l=3$ and $x_2=x=0$ is most sensitive.
However, the equation also holds by regarding $[x_1, 0]=[\varepsilon]=1$).
By Lemma \ref{Lem_Special Matrix} (8), we have
\[
CF(x_0, x_1, \ldots, x_{l-2}, x_{l-1}+x)=\frac{[x_0, x_1, \ldots, x_{l-2}]x+[x_0, x_1, \ldots, x_{l-2}, x_{l-1}]}{[x_1, x_2, \ldots, x_{l-2}]x+[x_1, x_2, \ldots, x_{l-2}, x_{l-1}]}.
\]
By substituting $x_i \,\, (i=0, 1, \cdots, l-1)$ to $k_i$, and Theorem \ref{Lem_4FRC_Matrix}, we have the result.
\end{proof}

\begin{Def} \label{SJEDGFSTYT}
{\rm
(1) \, For a finite design 
$
D=\{m\}_n=
  \underbrace{1 \cdots 11}_{\tiny k_0}
  \underbrace{0 \cdots 00}_{\tiny k_1}
  \cdots
  \underbrace{0 \cdots 00}_{\tiny k_{l-2}}
  \underbrace{1 \cdots 11}_{\tiny k_{l-1}}
$
and a finite or infinite design 
$
D'=
  \underbrace{1 \cdots 11}_{\tiny s_0}
  \underbrace{0 \cdots 00}_{\tiny s_1}
  \cdots
  \underbrace{1 \cdots 11}_{\tiny s_{t-1}}
  \cdots \cdots,
$
we define the {\color{black}{\it composition}} of $D$ and $D'$ by
\[
D \cdot D'=DD'=
  \underbrace{1 \cdots 11}_{\tiny k_0}
  \underbrace{0 \cdots 00}_{\tiny k_1}
  \cdots
  \underbrace{0 \cdots 00}_{\tiny k_{l-2}}
  \underbrace{1 \cdots 11}_{\tiny k_{l-1}}
  \underbrace{1 \cdots 11}_{\tiny s_0}
  \underbrace{0 \cdots 00}_{\tiny s_1}
  \cdots
  \underbrace{1 \cdots 11}_{\tiny s_{t-1}}
  \cdots \cdots.
\]
We note that $\theta_{DD'}=\theta_D+2^{-n}\theta_{D'}$.
If $D'$ is a finite design such that $D'=\{m'\}_{n'}$, then $D \cdot D'=DD'=\{2^{n'}m+m'\}_{n+n'}$. \\
(2) \, For a finite design 
$
D=\{m\}_n=
  \underbrace{1 \cdots 11}_{\tiny k_0}
  \underbrace{0 \cdots 00}_{\tiny k_1}
  \cdots
  \underbrace{0 \cdots 00}_{\tiny k_{l-2}}
  \underbrace{1 \cdots 11}_{\tiny k_{l-1}}
$
and a terminal design 
$
D'=\{2^{n'}\}_{n'}=
  1
  \underbrace{0 \cdots 00}_{\tiny n'}{}_{\tau},
$
we define the {\color{black}{\it composition}} of $D$ and $D'$ by
\[
\begin{split}
D \cdot D'=\{2^{n'}m+2^{n'}\}_{n+n'}&=(D+1)\underbrace{0 \cdots 00}_{\tiny n'} \\
&= 
\left\{
\begin{split}
&   \underbrace{1 \cdots 11}_{\tiny k_0}
     \underbrace{0 \cdots 00}_{\tiny k_1}
     \cdots
     \underbrace{0 \cdots 0}_{\tiny k_{l-2}-1}
     1
     \underbrace{0 \cdots 00}_{\tiny k_{l-1}}
     \underbrace{0 \cdots 00}_{\tiny n'} \,\,\,\, (m \neq 2^n-1), \\
&   1\underbrace{0 \cdots 00}_{n}\underbrace{0 \cdots 00}_{n'}{}_{\tau}=\{2^{n+n'}\}_{n+n'} \,\,\,\, (m=2^n-1).
\end{split}
\right. 
\end{split}
\]
}
\end{Def}
We note that for two finite designs $D$ and $D'$, we have $D(D'+1)=DD'+1$.
The following is the first main theorem of this paper.

\begin{Thm}(Design Composition Theorem I) \\
\label{Thm_The Deviding Formula of Ichio Function}
Let $D$ be a finite design and $D'$ an arbitrary design.
Then we have 
\[
\mathcal{A}(\theta_{DD'})=U(D) \mathcal{A}(\theta_{D'}).
\]
\end{Thm}

\begin{proof}
Firstly, we suppose $D'$ is a finite design.
Let
$
D=
  \underbrace{1 \cdots 11}_{\tiny k_0}
  \underbrace{0 \cdots 00}_{\tiny k_1}
  \cdots
  \underbrace{0 \cdots 00}_{\tiny k_{l-2}}
  \underbrace{1 \cdots 11}_{\tiny k_{l-1}}
$
and
$
D'=
  \underbrace{1 \cdots 11}_{\tiny s_0}
  \underbrace{0 \cdots 00}_{\tiny s_1}
  \cdots
  \underbrace{0 \cdots 00}_{\tiny s_{t-2}}
  \underbrace{1 \cdots 11}_{\tiny s_{t-1}}
$
be finite design representations. 
Then we have
\[
DD'=
  \underbrace{1 \cdots 11}_{\tiny k_0}
  \underbrace{0 \cdots 00}_{\tiny k_1}
  \cdots
  \underbrace{1 \cdots 11}_{\tiny k_{l-1}+s_0}
  \cdots
  \underbrace{0 \cdots 00}_{\tiny s_{t-2}}
  \underbrace{1 \cdots 11}_{\tiny s_{t-1}}.
\]
Since
\[
\mathcal{A}(\theta_{D'})=CF(s_0, s_1, \ldots, s_{t-2}, s_{t-1}) 
\]
and
\[
\mathcal{A}(\theta_{DD'})=CF(k_0, k_1, \ldots, k_{l-2}, k_{l-1}+s_0, s_1, \ldots, s_{t-2}, s_{t-1}), 
\]
we have
\[
\mathcal{A}(\theta_{DD'})=CF(k_0, k_1, \ldots, k_{l-2}, k_{l-1}+\mathcal{A}(\theta_{D'}))=U(D)\mathcal{A}(\theta_{D'}) 
\]
by Lemma \ref{DDUGEADXC}. 

For the case that $D'$ is an infinite design, by the continuity of the assembly function (cf. Theorem \ref{KIJLMMKU}), we have the result. 

Also for the case that $D'$ is a terminal design, we can confirm the same result by easy calculation.
\end{proof}

\begin{Exm} \label{LKHBYT}
{\rm We calculate the values of the assembly function.
\[
\begin{split}
(1) \,\, X&=\mathcal{A} (0.\dot{1}\dot{0}10 \cdots)
=U(10)\mathcal{A} (0.\dot{1}\dot{0}10 \cdots)
=\left(
\begin{matrix}
[11]  & [10] \\
[01]  & [10]
\end{matrix}
\right)\!X
=\left(
\begin{matrix}
2  &  1 \\
1  &  1
\end{matrix}
\right)\!X \quad \quad \quad \quad  \\[3pt]
&=\cfrac{2X+1}{X+1}\,.\,\,\,\,\text{Hence,}\,\,\,X=\mathcal{A} (0.\dot{1}\dot{0}10 \cdots)=\cfrac{1+\sqrt{5}}{2}\,. 
\end{split}
\]
\[
\begin{split}
(2) \,\, X&=\mathcal{A} (0.\dot{1}0\dot{1}101 \cdots)
=U(101)\mathcal{A} (0.\dot{1}0\dot{1}101 \cdots)
=\left(
\begin{matrix}
[110]  & [101] \\
[010]  & [011]
\end{matrix}
\right)\!X
=\left(
\begin{matrix}
2  &  3 \\
1  &  2
\end{matrix}
\right)\!X \\[3pt]
&=\cfrac{2X+3}{X+2}\,.\,\,\,\,\text{Hence,}\,\,\,X=\mathcal{A} (0.\dot{1}0\dot{1}101 \cdots)=\sqrt{3}\,. 
\end{split}
\]}
\end{Exm}

The following is the second main theorem of this paper.

\begin{Thm}(Design Composition Theorem II) \\
\label{Thm_The Deviding Formula of Matrix}
For two finite designs  $D$ and $D'$,
\[
U(DD')=U(D)U(D').
\]
\end{Thm}

\begin{proof}
We suppose $D'=\{m'\}_{n'}$.
Then the first and second columns of $U(D')$ are
\[
\left(
\begin{matrix}
[2^{n'} : m'+1]   \\
[2^{n'} : 2^{n'}-(m'+1)]  
\end{matrix}
\right) \,\,\,\,\,
\text{and} \,\,\,\,\,
\left(
\begin{matrix}
[2^{n'} : m']   \\
[2^{n'} : 2^{n'}-m']  
\end{matrix}
\right)
\]
respectively. 
They are in $\varDelta$ (cf. Definition \ref{Def_Design of Matrix}) by Corollary \ref{CorCoprime} (2), 
and correspond to $\mathcal{A}(\theta_{D'+1})$ and $\mathcal{A}(\theta_{D'})$ respectively.
By Theorem \ref{Thm_The Deviding Formula of Ichio Function}, we have
\[
U(D)\mathcal{A}(\theta_{D'+1})=\mathcal{A}(\theta_{D(D'+1)})=\mathcal{A}(\theta_{DD'+1}) \,\,\,\,\,
\text{and} \,\,\,\,\,
U(D)\mathcal{A}(\theta_{D'})=\mathcal{A}(\theta_{DD'}),
\]
and they correspond to the first and second columns of $U(DD')$ respectively.
Therefore we have the result.
\end{proof}

\begin{Exm} \label{HNUVOPJ}
{\rm \[
\begin{split}
&U(10)U(101)
=
\left(
\begin{matrix}
[11]  & [10] \\
[01]  & [10]
\end{matrix}
\right)
\left(
\begin{matrix}
[110]  & [101] \\
[010]  & [011]
\end{matrix}
\right)
=
\left(
\begin{matrix}
2  &  1 \\
1  &  1
\end{matrix}
\right)
\left(
\begin{matrix}
2  &  3 \\
1  &  2
\end{matrix}
\right)
=
\left(
\begin{matrix}
5  &  8 \\
3  &  5
\end{matrix}
\right). \\[8pt]
&U(10101)
=
\left(
\begin{matrix}
[10110]  & [10101] \\
[01010]  & [01011] 
\end{matrix}
\right)
=
\left(
\begin{matrix}
5  &  8 \\
3  &  5
\end{matrix}
\right).\,\,\,
\text{Hence,}\,\,
U(10)U(101)=U(10101).
\end{split}
\]}
\end{Exm}

\begin{Def} \label{SHJWQDHRSR}
{\rm
For two positive real numbers $\omega$ and $\eta$, if there exists 
$
A=\left(
\begin{matrix}
a  & b \\
c  & d  
\end{matrix}
\right) \in GL(2 ; \mathbb{Z})
$
such that
\[
\eta=A\omega=\frac{a\omega+b}{c\omega+d},
\]
then $\omega$ and $\eta$ are {\it equivalent}.
More precisely, if $ad-bc=1 \, (A \in SL(2 ; \mathbb{Z}))$, then $\omega$ and $\eta$ are {\it positively equivalent},
and if $ad-bc=-1 \, (A \in GL(2 ; \mathbb{Z}) \setminus SL(2 ; \mathbb{Z}))$, then $\omega$ and $\eta$ are {\it negatively equivalent}.
We note that $GL(2 ; \mathbb{Z})=
SL(2 ; \mathbb{Z}) 
\amalg 
\left(
\begin{matrix}
0  & 1 \\
1  & 0  
\end{matrix}
\right)
SL(2 ; \mathbb{Z})
$. 
}
\end{Def}

\begin{Def} \label{KKPPETAAB}
{\rm
For an infinite design 
$
D=
  \underbrace{1 \cdots 11}_{\tiny k_0}
  \underbrace{0 \cdots 00}_{\tiny k_1}
  \cdots
  \underbrace{0 \cdots 00}_{\tiny k_{l-2}}
  \underbrace{1 \cdots 11}_{\tiny k_{l-1}}
  \cdots
$,
the {\color{black}{\it conjugate design}} of $D$ is 
$
\bar{D}=
  \underbrace{0 \cdots 00}_{\tiny k_0}
  \underbrace{1 \cdots 11}_{\tiny k_1}
  \cdots
  \underbrace{1 \cdots 11}_{\tiny k_{l-2}}
  \underbrace{0 \cdots 00}_{\tiny k_{l-1}}
  \cdots
$.
}
\end{Def}
Then, we have $\mathcal{A}(\theta_{\bar{D}})=1/\mathcal{A}(\theta_{D})$ by Lemma \ref{SKWIDDGJGE} and Lemma \ref{EGHSRTWED}.

\begin{Cor} \label{JNPKGUNMHN}
Let $\omega$ and $\eta$ be two positive irrational numbers, 
and $\omega=\mathcal{A}(\theta)$ and $\eta=\mathcal{A}(\varphi)$.
Then we have the following: \\
$(1)$ \, $\omega$ and $\eta$ are positively equivalent if and only if there exist two finite designs $D_1$ and $D_2$,
and an infinite design $D$ such that $\theta=\theta_{D_1D}$ and $\varphi=\theta_{D_2D}$. \\
$(2)$ \, $\omega$ and $\eta$ are negatively equivalent if and only if there exist two finite designs $D_1$ and $D_2$,
and an infinite design $D$ such that $\theta=\theta_{D_1D}$ and $\varphi=\theta_{D_2\bar{D}}$.
\end{Cor}

\begin{proof}
Let $D$ and $D'$ be two infinite designs such that $\theta=\theta_D$ and $\varphi=\theta_{D'}$ respectively. \\
(1) \, By the assumption, there exists 
$
X=\left(
\begin{matrix}
a  & b \\
c  & d  
\end{matrix}
\right) \in SL(2 ; \mathbb{Z})
$
such that 
\[
\eta=X\omega=\frac{a\omega+b}{c\omega+d}.
\]
If necessary, by multiplying
$
\left(
\begin{matrix}
-1  & 0 \\
0  & -1  
\end{matrix}
\right),
$
we may assume that $a\omega+b>0$ and $c\omega+d>0$. \\
Let $D_n$ be the $n$-th segment of $D$ (cf. Definition \ref{AELMUQSST}), and $E_n$ an infinite design such that $D=D_nE_n$.
Since $\lim_{n \to \infty}\theta_{D_n}=\lim_{n \to \infty}\theta_{D_n+1}=\theta$ and Theorem \ref{KIJLMMKU} (continuity of the assembly function),
there exists a sufficiently large integer $N$ such that for $n \ge N$, we have
\[
\left(
\begin{matrix}
a  & b \\
c  & d  
\end{matrix}
\right)
\mathcal{A}(\theta_{D_n+1}) > 0 \,\,\,\,\, \text{and} \,\,\,\,\, 
\left(
\begin{matrix}
a  & b \\
c  & d  
\end{matrix}
\right)
\mathcal{A}(\theta_{D_n}) > 0.
\]
Hence $XU(D_n) \in \varGamma$ for $n \ge N$, and then we have
\[
X\mathcal{A}(\theta)=(XU(D_n))\mathcal{A}(\theta_{E_n})
\]
by Theorem \ref{Thm_The Deviding Formula of Ichio Function}.
By replacing $X$ and $D$ with $XU(D_n)$ and $E_n$ respectively,
we may assume $X \in \varGamma$. 
Since $\varGamma$ is generated by 
$
\left(
\begin{matrix}
1  & 1 \\
0  & 1  
\end{matrix}
\right)
$
and
$
\left(
\begin{matrix}
1  & 0 \\
1  & 1  
\end{matrix}
\right)
$
as a monoid, we just look at effects by the generators. 
Suppose
$
\eta=
\left(
\begin{matrix}
1  & 1 \\
0  & 1  
\end{matrix}
\right)\omega.
$
Since
$
\left(
\begin{matrix}
1  & 1 \\
0  & 1  
\end{matrix}
\right)
=U(2^1 : 1)
$,
we have $D'=1D$ by Theorem \ref{Thm_The Deviding Formula of Ichio Function}. 
Suppose
$
\eta=
\left(
\begin{matrix}
1  & 0 \\
1  & 1  
\end{matrix}
\right)\omega.
$
Since
$
\left(
\begin{matrix}
1  & 0 \\
1  & 1  
\end{matrix}
\right)
=U(2^1 : 0)
$,
we have $D'=0D$ by Theorem \ref{Thm_The Deviding Formula of Ichio Function}. 
Therefore we have the result. \\
(2) \, By (1), we just look at effects by
$
\left(
\begin{matrix}
0  & 1 \\
1  & 0  
\end{matrix}
\right)
$.
Suppose that
$
\eta=
\left(
\begin{matrix}
0  & 1 \\
1  & 0  
\end{matrix}
\right) \omega=\cfrac{1}{\omega}\,.
$
Definition \ref{KKPPETAAB} implies $D'=\bar{D}$. 
Therefore we have the result.
\end{proof}

\begin{Rem} \label{DWUORQHRD}
{\rm
(1) The proof of Corollary \ref{JNPKGUNMHN} shows us that for any matrix $M \in SL(2 ; \mathbb{Z})$, there exist two matrices $A, B \in \varGamma$ such that $M=AB^{-1}$ or $M=-AB^{-1}$. \\
(2) For two positive rational numbers $a/b=\mathcal{A}(\theta_{D_1})$ and $p/q=\mathcal{A}(\theta_{D_2})$, 
where $D_1$ and $D_2$ are reduced designs, we have
$
D_1=D_1\varepsilon, \,\, D_2=D_2\varepsilon \,\, \text{and} \,\, D_2=(D_2-1)\bar{\varepsilon}.
$
Hence $a/b$ and $p/q$ are positively and negatively equivalent.\\
(3) Let $?(x)$ be the Minkowski's question mark function (\cite{Conley}).
Then it is not so hard to see that
$$\frac{1}{\mathcal{A}(\theta)}=\frac{1}{?^{-1}(\theta)}-1.$$
}
\end{Rem}

\section{Periodic designs and quadratic numbers}

In this section, we define a {\it periodic design}, and
study relationship between periodic designs and quadratic irrational numbers 
by using the assembly function.

\begin{Def}
\label{Def_Design of Real Number}
{\rm Since  $\mathcal{A} :$ [0,1] $\rightarrow$  [0,$\infty$] is a bijective function,
for any non-negative real number $\omega$, there exists a real number $\theta \in [0, 1]$ such that $\omega=\mathcal{A}(\theta)$,
and a design $D$ such that $\theta=\theta_D$.
We note that If $\theta$ is of the form $m/2^n$, a reduced design $D$ is uniquely determined,
and if $\theta$ is not of the form $m/2^n$, an infinite design $D$ is uniquely determined.
Then we call $D$ the {\it design} of $\omega$.
}
\end{Def}

\begin{Def} \label{VXGMQEWAG}
{\rm
For an infinite design $D$, if there exist two finite designs $D'$ and $P \neq \varepsilon$ such that $D=D'PP \cdots$,
then we call $D$ a {\it periodic design} and $P$ a {\it period} of $D$.
For a periodic design $D$, we note the following: \\
(1) \, $D'$ and $P$ are not uniquely determined. \\
(2) \, $P$ can be replaced with a cyclic permutated one. \\
(3) \, $P$ can be replaced with a multiple of $P$. \\
(4) \, If there exists $P'$ such that $P$ is a multiple of $P'$, then $P'$ is also a period. \\
(5) \, For a period $P$, there exist two positive integers $m$ and $n$ such that $P=\{m\}_n$, $n \ge 2$ and $1 \le m \le 2^n-2$.

For a periodic design, a period with the minimal length is a {\it minimal period}.
If we can take $D'=\varepsilon$, then $D$ is called a {\it purely periodic design}.
We note that a purely periodic design can be denoted by $D=PPP \cdots$, where $P$ is the uniquely determined minimal period.
}
\end{Def}

\begin{Lem} \label{XTRGJEDN}
Let $D$ be a design and $\theta_D$ the design decimal of $D$.
Then we have the following: \\
{\rm (1)} \, $D$ is a finite design if and only if $\theta_D=m/2^n$, where $m$ and $n$ are integers such that $n \ge 0$ and $0 \le m \le 2^n-1$. \\
{\rm (2)} \, $D$ is a periodic design if and only if $\theta_D=m/\{2^k(2^n-1)\}$,
where $m$, $n$ and $k$ are integers such that $n \ge 2$, $0<m<2^k(2^n-1)$ and $k \ge 0$,
which implies $\theta_D$ is a rational number other than $m/2^n$. \\
{\rm (3)} \, $D$ is a purely periodic design if and only if $\theta_D=m/(2^n-1)$, where $m$ and $n$ are integers such that $n \ge 2$ and $0<m<2^n-1$.
Then the period is $\{m\}_n$.
\end{Lem}

\begin{proof}
(1) \, It is clear (cf. Definition \ref{AELMUQSST}). \\
(2) \, Let $D$ be a periodic design.
Then there exist two finite designs $D'$ and $P \neq \varepsilon$ such that $D=D'PP \cdots$.
Suppose that $D'=\{m'\}_k$ and $P=\{m''\}_n$, where $0 \le m' \le 2^k-1, 0 < m'' < 2^n-1$ and $n \ge 2$.
Then we have
\[
2^n\theta_D-\theta_D=(2^n-1)\theta_D=\frac{(2^n-1)m'+m''}{2^k}.
\] 
By setting $m=(2^n-1)m'+m''$, 
we have $0 < m < 2^k(2^n-1)$ and $\theta_D=m/\{2^k(2^n-1)\}$. 

Conversely, suppose $\theta_D=m/\{2^k(2^n-1)\}$, where $m$, $n$ and $k$ are integers such that $n \ge 2$, $0<m<2^k(2^n-1)$ and $k \ge 0$.
We can uniquely determine two integers $m' \, (0 \le m' \le 2^k-1)$ and $m'' \, (0 < m'' < 2^n-1)$ such that $m=(2^n-1)m'+m''$.
If we denote
$
D'=\{m'\}_k=
  \underbrace{1 \cdots 1}_{\tiny s_0}
  \underbrace{0 \cdots 0}_{\tiny s_1}
  \cdots
  \underbrace{1 \cdots 1}_{\tiny s_{t-1}}
$
and
$
P=\{m''\}_n=
  \underbrace{1 \cdots 1}_{\tiny k_0}
  \underbrace{0 \cdots 0}_{\tiny k_1}
  \cdots
  \underbrace{1 \cdots 1}_{\tiny k_{l-1}},
$
then
\[
\begin{split}
\theta_D&=\frac{m}{2^k(2^n-1)}=\frac{(2^n-1)m'+m''}{2^k(2^n-1)}=\frac{m'}{2^k}+\frac{1}{2^k}\frac{m''}{2^n-1}=\frac{m'}{2^k}+\frac{1}{2^k}\frac{m''/2^n}{1-1/2^n} \\[8pt]
&=0.
  \underbrace{1 \cdots 1}_{\tiny s_0}
  \underbrace{0 \cdots 0}_{\tiny s_1}
  \cdots
  \underbrace{1 \cdots 1}_{\tiny s_{t-1}}
  \underbrace{1 \cdots 1}_{\tiny k_0}
  \underbrace{0 \cdots 0}_{\tiny k_1}
  \cdots
  \underbrace{1 \cdots 1}_{\tiny k_{l-1}}
  \underbrace{1 \cdots 1}_{\tiny k_0}
  \underbrace{0 \cdots 0}_{\tiny k_1}
  \cdots
  \underbrace{1 \cdots 1}_{\tiny k_{l-1}}
  \cdots \cdots,
\end{split}
\]
which implies $D$ is a periodic design. 

For any rational number $\theta \, (0 \le \theta \le 1)$ other than $m/2^n$, we can denote it by $\theta=p/(2^kq)$ with $k \ge 0$, 
where $p$ and $2^kq$ are positive coprime integers and $q \, (\ge 3)$ is odd. 
Since 2 and $q$ are coprime, by the Fermat's little theorem, we have $2^{\varphi(q)} \equiv 1 \, (\text{mod}\,q)$.
Hence there exists a positive integer $a$ such that $2^{\varphi(q)}-1=qa$, 
and $\theta=p/(2^kq)=pa/(2^kqa)=pa/(2^k(2^{\varphi(q)}-1))$. 
Hence a rational number except $m/2^n$ is of the form $m/\{2^k(2^n-1)\}$. \\
(3) \, Under the situation $D'=\varepsilon$ or $k=0$, we have the result in the same way as (2).
\end{proof}

\begin{Thm}
\label{Thm Ichio Function}
Let $\omega=\mathcal{A}(\theta) \,\, (\theta \in [0, 1))$ be the assembly function. \\
{\rm (1)} $\omega$ is a rational number if and only if $\theta=m/2^n \,\, (0 \le m \le 2^n-1)$, namely $\theta$ is a finite design decimal. \\
{\rm (2)} $\omega$ is a quadratic irrational number if and only if $\theta$ is a rational number other than $m/2^n$,
namely $\theta$ is a periodic design decimal. \\ 
\end{Thm}

\begin{proof}
(1) It is clear from the classical theory of continued fractions (cf. \cite{Hardy}) and the definition of the assembly function. \\
(2) Suppose that $\omega$ is a quadratic irrational number.
Since a quadratic irrational number can be expressed by a periodic continued fraction (cf. \cite{Hardy}),
there exists a periodic design $D$ such that $\theta=\theta_D$ by Lemma \ref{SKWIDDGJGE}. 
Then $\theta$ is a rational number other than $m/2^n$ by Lemma \ref{XTRGJEDN} (2). 

Conversely, let $\theta$ be a rational number except $m/2^n$.
By Lemma \ref{XTRGJEDN} (2), there exists a periodic design $D$ sch that $\theta=\theta_D$.
Then there exist two finite designs $D'$ and $P$ such that $D=D'PP\cdots$.
Since $D$ is an infinite design, $\omega=\mathcal{A}(\theta_D)$ is an irrational number.
By Theorem \ref{Thm_The Deviding Formula of Ichio Function}, we have
\[
\omega=\mathcal{A}(\theta_D)=\mathcal{A}(\theta_{D'PP\cdots})=U(D')U(P)U(D')^{-1}\mathcal{A}(\theta_D)=U(D')U(P)U(D')^{-1}\omega.
\]
Hence, $\omega$ is a quadratic irrational number. \\
\end{proof}

\begin{Def} \label{DHHUDDGFGJ}
For $\omega=\mathcal{A}(\theta) \,\, (\theta \in [0, 1))$, if there exists a purely periodic design $D$ such that $\theta=\theta_D$, 
we call $\omega$ a {\color{black}{\it pure quadratic irrational number}}.
If $\omega$ is a quadratic irrational number but not a pure quadratic irrational number,
we call $\omega$ a {\color{black}{\it non-pure quadratic irrational number}}. 
\end{Def}

\begin{Thm}
\label{SGFTJAKRE}
Let $\omega=\mathcal{A}(\theta) \,\, (\theta \in [0, 1))$ be the assembly function. \\
{\rm (1)} $\omega$ is a pure quadratic irrational number if and only if $\theta=m/(2^n-1) \, (1 \le m \le 2^n-2)$. \\ 
{\rm (2)} Suppose $\theta=p/q \, (0 \le \theta \le 1)$, where $p$ and $q$ are positive coprime integers and $\theta \neq m/2^n \, (0 \le m \le 2^n)$.
Then, we have the following: \\
{\rm (i)}  $\omega$ is a pure quadratic irrational number if and only if $q$ is odd. \\ 
{\rm (ii)} $\omega$ is a non-pure quadratic irrational number if and only if $q$ is even. \\ 
{\rm (3)} $\omega$ is a purely periodic continued fraction if and only if $\theta=m/(2^n-1)$, where $m$ is an even integer such that $2^{n-1}\le m \le2^n-2$.
\end{Thm}

\begin{proof}
(1) By Lemma \ref{XTRGJEDN} (3), we have the result.\\
(2) (i) Suppose $\omega=\mathcal{A}(\theta)$ is a pure quadratic irrational number.
Then there exists a purely periodic design $D$ such that $\theta=\theta_D$.
By Lemma \ref{XTRGJEDN} (3), we have $\theta_D=m/(2^n-1) \,\, (1<m<2^n-1)$.

Conversely, since 2 and $q$ are coprime, by the Fermat's little theorem, we have $2^{\varphi(q)} \equiv 1 \, (\text{mod}\,q)$.
Hence, there exists a positive integer $a$ such that $2^{\varphi(q)}-1=qa$. 
Then $\theta=p/q=pa/qa=pa/(2^{\varphi(q)}-1)$. 
By Lemma \ref{XTRGJEDN} (3), there exists a purely periodic design $D$ such that $\theta=\theta_D$.
Hence, $\omega$ is a pure quadratic irrational number. \\
(2) (ii) By Lemma \ref{XTRGJEDN} (2) and (2) (i), we have the result. \\
(3) Suppose $\omega=\mathcal{A}(\theta)=CF(k_0, k_1, \ldots, k_{l-1}, k_0, k_1, \ldots, k_{l-1}, \ldots)$ is a purely periodic continued fraction,
where $k_0,\,k_1,\ldots, k_{l-1}$ is the minimal period of quotients.
Then by Lemma \ref{SKWIDDGJGE}, there exists a purely periodic design $D=PP\cdots$ such that $\theta=\theta_D$.
Actually, the minimal period of design is given by
\[
\begin{split}
P 
&= 
\left\{
\begin{split}
&   \underbrace{1 \cdots 11}_{\tiny k_0}
     \underbrace{0 \cdots 00}_{\tiny k_1}
     \cdots
     \underbrace{1 \cdots 11}_{\tiny k_{l-2}}
     \underbrace{0 \cdots 00}_{\tiny k_{l-1}} \,\,\, ( \, l : \text{even}), \\
&   \underbrace{1 \cdots 11}_{\tiny k_0}
     \underbrace{0 \cdots 00}_{\tiny k_1}
     \cdots
     \underbrace{0 \cdots 00}_{\tiny k_{l-2}}
     \underbrace{1 \cdots 11}_{\tiny k_{l-1}}
     \underbrace{0 \cdots 00}_{\tiny k_0}
     \underbrace{1 \cdots 11}_{\tiny k_1}
     \cdots
     \underbrace{1 \cdots 11}_{\tiny k_{l-2}}
     \underbrace{0 \cdots 00}_{\tiny k_{l-1}} \,\,\, ( \, l : \text{odd}).
\end{split}
\right. 
\end{split}
\]
We set $P=\{m\}_n$.
Then $m$ is an even integer such that $2^{n-1} \le m \le 2^n-2$, and by Lemma \ref{XTRGJEDN} (3), we have the result.
\end{proof}

\begin{Lem} \label{LKBUYOIKGS}
Suppose $\omega=\mathcal{A}(\theta_D)$ is a pure quadratic irrational number,
where $D$ is a purely periodic design with the period $\{m\}_n \, (n \ge 2, \,\, 1 \le m \le 2^n-2)$.
Then $\omega$ is a root of 
\[
[2^n : 2^n-(m+1)]X^2-(\,[2^n : m+1]-[2^n : 2^n-m]\,)X-[2^n : m]=0. 
\]
The conjugate $\omega'$ of $\omega$ is negative,
and 
$
-\omega'=\mathcal{A}(\theta_{D^*}),
$
where $D^*$ is a purely periodic design with the period $\{m^*\}_n$ which is the inverse design of $\{m\}_n$
(cf. Definition \ref{ASWQHOPM}).
\end{Lem}

\begin{proof}
Since $\omega$ is a irrational number by Theorem \ref{Thm Ichio Function},
$\omega'$ is also a irrational number.
By Theorem \ref{Thm_The Deviding Formula of Ichio Function} and the pure periodicity, we have
\[
\omega=U(2^n : m)\omega=\frac{[2^n : m+1]\omega+[2^n : m]}{[2^n : 2^n-(m+1)]\omega+[2^n : 2^n-m]}.
\]
Hence, $\omega$ is a root of  
\[
[2^n : 2^n-(m+1)]X^2-(\,[2^n : m+1]-[2^n : 2^n-m]\,)X-[2^n : m]=0, 
\]
and by the Vieta's formulas, $\omega'$ is a negative root of the equation. \\
By Theorem \ref{Thm_m^*}, $-\mathcal{A}(\theta_{D^*})$ is a root of this equation. 
Since $-\mathcal{A}(\theta_{D^*}) <0$, we have $-\omega'=\mathcal{A}(\theta_{D^*})$.
\end{proof}

\begin{Lem} \label{SDEEAAUBJRD}
Suppose $D$ is a purely periodic design with the period: 
\[
P=\{m\}_n=
  \underbrace{1 \cdots 11}_{\tiny k_0}
  \underbrace{0 \cdots 00}_{\tiny k_1}
  \cdots
  \underbrace{0 \cdots 00}_{\tiny k_{l-2}}
  \underbrace{1 \cdots 11}_{\tiny k_{l-1}},
\]
where $k_0 \ge 0, k_i \ge 1 \,\, (i=1, 2, \cdots, l-2)$, $k_{l-1} \ge 0, n=\sum_{i=0}^{l-1}k_i \ge 2$ and $1 \le m \le 2^n-2$. 
Let $\omega=\mathcal{A}(\theta_D)$, and $\omega'$ the conjugate of $\omega$.
Then we have the following: \\
{\rm (1)} \, $k_0\ge1$ and $k_{l-1}= 0$ if and only if $\omega>1$ and $-1<\omega'<0$. \\
{\rm (2)} \, $k_0\ge1$ and $k_{l-1}\ge1$ if and only if $\omega>1$ and $\omega'<-1$. \\
{\rm (3)} \, $k_0 = 0$  and $k_{l-1}=0$ if and only if $0<\omega<1$ and $-1<\omega'<0$. \\ 
{\rm (4)} \, $k_0 = 0$  and $k_{l-1}\ge1$ if and only if $0<\omega<1$ and $\omega'<-1$. 
\end{Lem}

\begin{proof}
By Lemma \ref{EGHSRTWED} (1) and Lemma \ref{LKBUYOIKGS}, we have the result.
\end{proof}

If $\omega$ satisfies one of (1), (2), (3) and (4) in Lemma \ref{SDEEAAUBJRD}, 
we call $\omega$ a quadratic irrational number of 
{\it type $1$}, {\it type $2$}, {\it type $3$} and {\it type $4$} respectively.

\begin{Thm} \label{LMIOUYDEES}
Let $\omega=\mathcal{A}(\theta)$ be a positive quadratic irrational number and $\omega'$ the conjugate of $\omega$.
Then we have the following: \\
{\rm (1)} \, $\omega$ is a pure quadratic irrational number if and only if $\omega'<0$. \\
{\rm (2)} \, $\omega$ is a non-pure quadratic irrational number if and only if $\omega'>0$. \\
\end{Thm}

\begin{proof}
It is sufficient to show ``if parts'' of (1) and (2). \\
(1) \, Suppose that $\theta$ is a purely periodic design decimal.
Then by Lemma \ref{LKBUYOIKGS}, we have $\omega'<0$. \\
(2) \, Suppose that $\theta$ is a periodic design decimal but is not a purely periodic design decimal, and that $\theta$ is represented by an infinite design $D$.
Then there exist a finite design $D' \neq \varepsilon$ and a period $P \neq \varepsilon$ such that $D=D'PP \cdots$.
We set $D'=d\,'_1d\,'_2 \cdots d\,'_k$ and $P=p_1p_2 \cdots p_n$.
Then we may suppose $d\,'_k \neq p_n$.
We set $D''=d\,'_1d\,'_2 \cdots d\,'_{k-1}, \, E=PP \cdots, \eta=\mathcal{A}(\theta_E)$, and the conjugate root of $\eta$ as $\eta'$.
We note that for any finite design $F$, the conjugate root of $\mathcal{A}(\theta_{FE})=U(F)\eta$ is $U(F)\eta'$. \\
Case 1: $(d\,'_k, p_n)=(0, 1)$. 
By Lemma \ref{SDEEAAUBJRD} (1) and (3), we have $\eta'<-1$. Since
\[
U(0)\eta'=
\left(
\begin{matrix}
1  & 0 \\
1  & 1  
\end{matrix}
\right)\eta'
=\frac{\eta'}{\eta'+1}=\frac{1}{1+1/\eta'}>0
\]
and $D'=D''0$, we have $\omega'=U(D'')\{\eta'/(\eta'+1)\} >0$ by Theorem \ref{Thm_The Deviding Formula of Ichio Function}. \\
Case 2: $(d\,'_k, p_n)=(1, 0)$. 
By Lemma \ref{SDEEAAUBJRD} (2) and (4), we have $-1<\eta'<0$. Since
\[
U(1)\eta'=
\left(
\begin{matrix}
1  & 1 \\
0  & 1  
\end{matrix}
\right)\eta'
=\eta'+1>0
\]
and $D'=D''1$, we have  $\omega'=U(D'')(\eta'+1) >0$ by Theorem \ref{Thm_The Deviding Formula of Ichio Function}. 
Therefore this completes the proof.
\end{proof}

A.M.Legendre determined the continued fraction expression of the square root of a rational number.
We obtain an essentially equivalent result via design as follows.

\begin{Cor} \label{DDDWHKHCDGH}
For a positive rational number $Q$, $\sqrt{Q}$ is a quadratic irrational number if and only if 
there exists a purely periodic design $D$ with a palindromic period
such that $\sqrt{Q}=\mathcal{A}(\theta_D)$.
\end{Cor}

\begin{proof}
Since the conjugate of $\sqrt{Q}$ \, is $-\sqrt{Q}$, by Theorem \ref{LMIOUYDEES}, $\sqrt{Q}$ is represented by a purely periodic design. 
By Definition \ref{ASWQHOPM}, 
Theorem \ref{KIJLMMKU} and Lemma \ref{LKBUYOIKGS}, we have the result.
\end{proof}

\begin{Exm} \label{Symmetric Minimal Period} 
\[
\begin{split}
(1)& \sqrt{2}=\mathcal{A}(
0.
\dot{1}00\dot{1}1001 \cdots). \qquad \qquad \quad \,\,\,\,\,
(2) \, \sqrt{3}=\mathcal{A}(
0.
\dot{1}0\dot{1}101 \cdots). \\
(3)& \sqrt{5}=\mathcal{A}(
0.
\dot{1}100001\dot{1}11000011 \cdots).  \quad \,\,\,\,\,
(4) \, \sqrt{6}=\mathcal{A}(
0.
\dot{1}1001\dot{1}110011 \cdots). \\
(5)& \sqrt{7}=\mathcal{A}(
0.
\dot{1}10101\dot{1}1101011 \cdots). \quad \quad \,\,\,\,\,
(6) \, \sqrt{8}=\mathcal{A}(
0.
\dot{1}101\dot{1}11011 \cdots). \\
(7)& \sqrt{1/3}=\mathcal{A}(
0.
\dot{0}1\dot{0}010 \cdots).  \qquad \qquad \qquad \!
(8) \, \sqrt{2/5}=\mathcal{A}(
0.
\dot{0}101101\dot{0}01011010 \cdots).
\end{split}
\]
\end{Exm}

\section{Differentiability of the assembly function}

In this section, as an application of previous sections, 
we discuss differentiability of the assembly function at rational points.

\medskip

For the assembly function $w=\mathcal{A}(\theta)\ (\theta \in (0, 1))$, if
$$\lim_{h \to -0}\frac{\mathcal{A}(\eta+h)-\mathcal{A}(\eta)}{h}
=\lim_{h \to +0}\frac{\mathcal{A}(\eta+h)-\mathcal{A}(\eta)}{h}=+\infty
\quad (\eta \in (0, 1)),$$
then we denote by $\mathcal{A}'(\eta)=\infty$.
However $\mathcal{A}(\theta)$ is NOT differential at $\theta=\eta$.

\begin{Lem}
\label{Fibonacci}
Let $b_m=[
\underbrace{1, 1, \ldots, 1}_{\tiny m}]\ (m=1, 2, \ldots)$ be the continuant 
(cf.\ Definition \ref{Def_Special Matrix}).
Then we have
$$b_m=\frac{5+2\sqrt{5}}{5}\left(
\frac{1+\sqrt{5}}{2}
\right)^{m-2}
+\frac{5-2\sqrt{5}}{5}\left(
\frac{1-\sqrt{5}}{2}
\right)^{m-2},$$
and
$$\left(
\frac{1+\sqrt{5}}{2}
\right)^{m-2}<b_m.$$
\end{Lem}

\begin{proof}
By Definition \ref{Def_Special Matrix} (\ref{SXSAWJTDDJ}), 
we have $b_1=1$, $b_2=2$ and $b_{m+2}=b_{m+1}+b_m$ 
(i.e.\ $\{b_m\}_{m=1}^{\infty}$ is the Fibonacci's sequence).
Hence we have the result.
\end{proof}

\begin{Lem}
\label{diff=0}
$\mathcal{A}'(2/3)=0$.
\end{Lem}

\begin{proof}
The binary decimal of $2/3$ is $\eta=0.\dot{1}\dot{0}10\cdots$.

\medskip

\noindent
(1) On the right derivative:

We set
$\theta_n=0.\underbrace{10\cdots 10}_{\tiny 2n-2},\ 
\rho_n=0.\underbrace{10\cdots 10}_{\tiny 2n-4}11$,\ 
and
$h_n=2^{-2n}.$
Note that
$$\eta+h_n=0.\underbrace{10\cdots 10}_{\tiny 2n-2}11\dot{1}\dot{0}10\cdots.$$
Then we have
$\theta_n<\eta<\eta+h_n<\rho_n$.

\medskip

Suppose that $h_{n-1}<h\le h_n$.
Then by Theorem \ref{KIJLMMKU}, we have
\begin{align}
0<\frac{\mathcal{A}(\eta+h)-\mathcal{A}(\eta)}{h}
<\frac{\mathcal{A}(\rho_n)-\mathcal{A}(\theta_n)}{h_{n-1}}. \tag{8:1} \label{8:1} 
\end{align}

Suppose that the design of $\theta_n$ is $\{m\}_{2n-2}$.
Then that of $\rho_n$ is $\{m+1\}_{2n-2}$.
By Theorem \ref{Thm_det=1} (2:1), Definition \ref{Def_Ichio Function} and Theorem \ref{KIJLMMKU}, 
we have
\begin{align}
\mathcal{A}(\rho_n)-\mathcal{A}(\theta_n)
=\frac{1}{[2^{2n-2}:2^{2n-2}-(m+1)][2^{2n-2}:2^{2n-2}-m]}. \tag{8:2} \label{8:2} 
\end{align}
Since
$\{2^{2n-2}-(m+1)\}_{2n-2}
=\underbrace{01\cdots 01}_{\tiny 2n-2}\quad
\mbox{and}\quad
\{2^{2n-2}-m\}_{2n-2}
=\underbrace{01\cdots 01}_{\tiny 2n-4}10$, 
we have 
$$
[2^{2n-2}:2^{2n-2}-(m+1)]
=[\underbrace{1, 1, \cdots, 1}_{\tiny 2n-3}]
=b_{2n-3}
$$
and
\begin{align}
[2^{2n-2}:2^{2n-2}-m]
=[\underbrace{1, 1, \cdots, 1}_{\tiny 2n-6}, 2]
=b_{2n-5}+b_{2n-6}=b_{2n-4}\tag{8:3} \label{8:3} 
\end{align}
by Theorem \ref{Thm_FRC_Matrix}, 
Lemma \ref{Lem_Special Matrix} (6), (7), (8) and Lemma \ref{Fibonacci}.

\medskip

By (8:1), (8:2) and (8:3), we have
$$0<\frac{\mathcal{A}(\eta+h)-\mathcal{A}(\eta)}{h}<
\frac{2^{n-1}}{b_{2n-3}}\times \frac{2^{n-1}}{b_{2n-3}}.$$
Since
$$\left(
\frac{1+\sqrt{5}}{2}
\right)^2=\frac{3+\sqrt{5}}{2}>2,$$
we have
$$\lim_{n\to \infty}\left(\frac{2^{n-1}}{b_{2n-3}}\times \frac{2^{n-1}}{b_{2n-3}}\right)=0\quad
\mbox{and}\quad
\lim_{h\to +0}\frac{\mathcal{A}(\eta+h)-\mathcal{A}(\eta)}{h}=0.$$

\medskip

\noindent
(2) On the left derivative:

By the same way as (1), we have
${\displaystyle
\lim_{h\to -0}\frac{\mathcal{A}(\eta+h)-\mathcal{A}(\eta)}{h}=0}$.

Therefore we have the result.
\end{proof}

\begin{Lem}
\label{affine}
Let $D=\{m\}_n\ (0\le m\le 2^n-1)$ be a finite design, and
$D'$ an arbitrary design.
Then we have the following:

\medskip

\noindent
{\rm (1)}\ 
$$\mathcal{A}'(\theta_{DD'})
=\frac{2^n\mathcal{A}'(\theta_{D'})}
{\left\{[2^n:2^n-(m+1)]\mathcal{A}(\theta_{D'})+[2^n:2^n-m]
\right\}^2}.$$

\noindent
{\rm (2)}\ 
$\mathcal{A}'(\theta_{DD'})=0$ if and only if $\mathcal{A}'(\theta_{D'})=0$.

\medskip

\noindent
{\rm (3)}\ 
$\mathcal{A}'(\theta_{DD'})=\infty$ if and only if $\mathcal{A}'(\theta_{D'})=\infty$.

\end{Lem}

\begin{proof}
(1) By Theorem \ref{Thm_det=1} (2:1), 
Theorem \ref{Thm_The Deviding Formula of Ichio Function} 
and $\theta_{D'}=2^n\theta_{DD'}-m$, 
we have
$$\mathcal{A}(\theta_{DD'})=
\frac{[2^n:m+1]\mathcal{A}(\theta_{D'})+[2^n:m]}
{[2^n:2^n-(m+1)]\mathcal{A}(\theta_{D'})+[2^n:2^n-m]},$$
and
$$\mathcal{A}'(\theta_{DD'})=
\frac{2^n\mathcal{A}'(\theta_{D'})}
{\left\{[2^n:2^n-(m+1)]\mathcal{A}(\theta_{D'})+[2^n:2^n-m]\right\}^2}.$$

\noindent
(2), (3) 
Since (1) and $[2^n:2^n-m]>0$, we have the result.
\end{proof}

\newpage

The following is the third main theorem.

\begin{Thm}
\label{main3}
Let $S$ be the set of rational numbers of the form $m/2^n$, 
where $m$ and $n$ are two positive integers with $1 \le m \le 2^n-1$,
and $T$ the set of rational numbers in $(0,1)$.
Then for the assembly function $w=\mathcal{A}(\theta)$, we have the following:

\medskip

\noindent
{\rm (1)}\ 
For $\eta \in S$, we have $\mathcal{A}'(\eta)=\infty$.

\medskip

\noindent
{\rm (2)}\ 
For $\eta \in T\setminus S$, 
if $w=\mathcal{A}(\theta)$ is differential at $\theta=\eta$, then
we have $\mathcal{A}'(\eta)=0$.

\end{Thm}

\begin{proof}
(1) By Lemma \ref{affine} (3), it is sufficient to show only the case $\eta=1/2=0.1$.

\medskip

\noindent
(i) On the right derivative:

We set
$h_n=2^{-n-1}.$
Note that $\eta+h_n=0.1\underbrace{0\cdots 0}_{\tiny n-1}1.$

\medskip

Suppose that $h_n<h\le h_{n+1}$.
Since $\mathcal{A}(1/2)=1$ (Lemma \ref{EGHSRTWED} (1)), 
$$U(\{1\underbrace{0\cdots 0}_{\tiny n-1}\}_n)
=U(\{1\}_1)U(\{0\}_1)^{n-1}
=\left(
\begin{matrix}
1  & 1 \\
0  & 1  
\end{matrix}
\right)
\left(
\begin{matrix}
1  & 0 \\
1  & 1  
\end{matrix}
\right)^{n-1}
=\left(
\begin{matrix}
n  & 1 \\
n-1  & 1  
\end{matrix}
\right)$$
(Theorem \ref{Thm_The Deviding Formula of Matrix})
and Theorem \ref{Thm_The Deviding Formula of Ichio Function}, 
we have
$$\mathcal{A}(\eta+h_n)=\frac{n+1}{n}.$$
Then by Theorem \ref{KIJLMMKU}, we have
$$
\frac{\mathcal{A}(\eta+h)-\mathcal{A}(\eta)}{h}
>\frac{\mathcal{A}(\eta+h_n)-\mathcal{A}(\eta)}{h_{n+1}}
=\frac{2^{n+2}}{n}.
$$
Since 
$$\lim_{n\to \infty}\frac{2^{n+2}}{n}=+\infty,$$
we have
$$\lim_{h\to +0}\frac{\mathcal{A}(\eta+h)-\mathcal{A}(\eta)}{h}=+\infty.$$

\medskip

\noindent
(ii) On the left derivative:

By the same way as (i), we have
${\displaystyle
\lim_{h\to -0}\frac{\mathcal{A}(\eta+h)-\mathcal{A}(\eta)}{h}=+\infty}$.

Therefore we have the result.

\bigskip

\noindent
(2) By Lemma \ref{affine} (2), it is sufficient to show only the case that
$\eta=\theta_D$ and $D$ is purely periodic with the period $P\ne \{0\}_1, \{1\}_1$.
We set $P=\{m\}_n$ with $1 \le m \le 2^n-1$ and $n>0$.
By Lemma \ref{affine} (1), we have
$$\mathcal{A}'(\eta)
=\frac{2^n\mathcal{A}'(\eta)}
{\left\{[2^n:2^n-(m+1)]\mathcal{A}(\eta)+[2^n:2^n-m]
\right\}^2}.$$
Suppose that $\mathcal{A}'(\eta)\ne 0$.
Then we have
\begin{align}
[2^n:2^n-(m+1)]\mathcal{A}(\eta)+[2^n:2^n-m]=2^{n/2}.
\tag{8:4} \label{8:4} 
\end{align}
By Theorem \ref{Thm_det=1} (2:1) and Lemma \ref{LKBUYOIKGS}, we have
\begin{align}
\mathcal{A}(\eta)
=\frac{[2^n : m+1]-[2^n : 2^n-m]
+\sqrt{([2^n : m+1]+[2^n : 2^n-m])^2-4}}{2[2^n : 2^n-(m+1)]}.
\tag{8:5} \label{8:5} 
\end{align}
By (8:4) and (8:5), $Y=[2^n : m+1]+[2^n : 2^n-m]$ satisfies
$$2^{n/2}=\frac{Y+\sqrt{Y^2-4}}{2}.$$
Since $Y$ is an integer, we have $Y=2$ and $n=0$.
This is a contradiction.
Therefore we have $\mathcal{A}'(\eta)=0$.
\end{proof}

\begin{Rem}
{\rm
It is easy to see from Lemma \ref{diff=0}, Lemma \ref{affine} and Theorem \ref{main3}, 
both the sets
$X_0=\{\eta \in (0,1)\ |\ \mathcal{A}'(\eta)=0\}$
and
$X_{\infty}=\{\eta \in (0,1)\ |\ \mathcal{A}'(\eta)=\infty\}$
are dense in $(0,1)$.
Lebesgue's theorem says that a monotonely increasing continuous function 
is differentiable at almost every points.
Since the assembly function is strictly increasing continuous, 
it is differentiable at almost every points.
In our next paper \cite{Yamada2}, we show that at every point,
the derivation is $0$ or $\infty$, and gave a necessary and sufficient condition
for differentiability on any rational point.
}
\end{Rem}





\bigskip

We have studied the fundamental properties of design and the assembly function.
In the forthcoming paper, on the basis of these knowledge, we shall clarify importance of 
the assembly function and the relationship with Markov conjecture.
\\ \\
\textbf{Acknowledgments.} 
I have studied SDIs, assembly function and Markov conjecture for almost 30 years.
Under the guidance of Akito Nomura, I have tried to complete the present paper since 6 years ago.
However, regrettably, Nomura passed away in June 2016.
Afterwards, under the guidance of Teruhisa Kadokami, I have completed the present paper.
For their useful advices, I am sincerely grateful to them.


\newpage

\begin{thebibliography}{9}
\bibitem{Calkin}  N. Calkin and H. S. Wilf, Recounting the rationals, this MONTHLY 107 (2000) 360-367.
\bibitem{Conley}  R. M. Conley, A Survey of the Minkowski ?(x) Function, M. S, thesis, West Virginia University, Morgantown, WV, 2003.
\bibitem{Giuli}  C. Giuli and R. Giuli, A primer on Stern's diatomic sequence I, Fibonacci Quart. 17 (1979) 103-108.
\bibitem{Hardy}  G. H. Hardy and E. M. Wright, An Introduction to the Theory of Numbers, Fifth Edition, Oxford. 
\bibitem{Northshield}  Sam Northshield, Stern's Diatomic Sequence 0,1,1,2,1,3,2,3,1,4, $\cdots$, Department of Mathematics, SUNY, Plattsburgh, NY 12901.
\bibitem{Stern}  M. A. Stern, $\ddot{\textrm{U}}$ber eine zahlentheoretische Funktion, J. Reine Agnew. Math. 55 (1858) 193-220.
\bibitem{Yamada2}  Y. Yamada, Differential and integral of the assembly function, preprint.
\end{thebibliography}
\end{document}